\documentclass[11pt]{article}

\usepackage{amsmath,amssymb,amsthm,mathtools}
\usepackage{geometry}
\usepackage{graphicx}
\usepackage{hyperref}
\usepackage{enumitem}
\usepackage[numbers]{natbib}
\usepackage{notoccite}
\usepackage{todonotes}
\usepackage{tikz}
\usetikzlibrary{positioning,calc,arrows.meta}
\newtheorem{theorem}{Theorem}[section]
\newtheorem{proposition}[theorem]{Proposition}
\newtheorem{lemma}[theorem]{Lemma}
\newtheorem{corollary}[theorem]{Corollary}
\theoremstyle{definition}
\newtheorem{definition}[theorem]{Definition}
\newtheorem{example}[theorem]{Example}
\theoremstyle{remark}
\newtheorem{remark}[theorem]{Remark}
\newtheorem{claim}[theorem]{Claim}

\newcommand{\CP}{\mathbb{CP}}
\newcommand{\Z}{\mathbb{Z}}
\newcommand{\R}{\mathbb{R}}
\newcommand{\C}{\mathbb{C}}
\newcommand{\WCP}{\mathbb{CP}}
\newcommand{\st}{\mathcal{S}}
\newcommand{\e}{\mathfrak{e}}
\newcommand{\sto}{\mathfrak{o}}

\title{Symplectic log Kodaira dimension $-\infty$, Hirzebruch--Jung strings and weighted projective planes}
\author{Tian-Jun Li and Shengzhen Ning}
\AtEndDocument{\bigskip{\footnotesize
		\textsc{School of Mathematics, University of Minnesota, Minneapolis, MN, US} \par
		\textit{E-mail address}: \texttt{tjli@math.umn.edu} \par
		\addvspace{\medskipamount}
		\textsc{School of Mathematics, University of Minnesota, Minneapolis, MN, US} \par
		\textit{E-mail address}: \texttt{ning0040@umn.edu} \par
		
}}
\date{\today}

\begin{document}
\maketitle

\begin{abstract}
We study symplectic minimal resolutions of weighted projective planes $\mathbb{CP}(a,b,c)$ from the perspective of disconnected symplectic divisors with symplectic log Kodaira dimension $-\infty$. Building on the techniques developed in our previous work \cite{LN25} for connected divisors, we introduce the notion of exceptional gaps between distinct connected components of the divisor and use it to establish a Torelli-type theorem for certain configurations of three Hirzebruch--Jung strings. Motivated by Daigle--Russell's study of affine rulings on complete normal rational surfaces in algebraic context (\cite{DP1,DP2}), we also establish a weighted version of Gromov--McDuff's characterization of symplectic $\mathbb{CP}^2$ by showing the existence of symplectic affine rulings implies certain divisor configuration to arise from the minimal resolution of $\mathbb{CP}(a,b,c)$.
\end{abstract}

\tableofcontents

\section{Introduction}

\subsection{Background} 
Let $(X,\omega)$ be a closed symplectic $4$-manifold. A configuration $D=\cup C_i\subseteq (X,\omega)$ of closed embedded symplectic surfaces is called a \emph{symplectic divisor} if every intersection is positive, transverse, and $\omega$-orthogonal, and if there are no triple intersections. Symplectic divisors in dimension $4$ have been studied extensively, beginning with Gromov's characterization of symplectic $S^2\times S^2$ (\cite{Gromov}), where $D$ is the union of two spheres with self-intersection $0$, and McDuff's characterization of symplectic rational ruled manifolds (\cite{Mcduffrationalruled}), where $D$ is a single sphere of nonnegative self-intersection. They provide a flexible analogue of normal crossing divisors in the algebraic setting, making it possible to import ideas from algebraic geometry into a context where no specific integrable complex structure is fixed. Foundational work of Donaldson (\cite{Donaldson}) and Taubes (\cite{Taubes}) offers a rich source of symplectic divisors that exist in a wide range of settings. At the same time, they play a central role in symplectic surgery constructions such as Gompf's symplectic sum (\cite{Gompf}) and Fintushel--Stern's rational blowdown (\cite{FintushelStern}). Since they also arise naturally as resolutions of surface singularities with contact links, they are closely related to the study of symplectic fillings (\cite{Lisca,OhtaOno,GayStipsicz}). Moreover, the study of topology of symplectomorphism groups is also closely related to the understanding of symplectic divisors (\cite{Abreu,AbreuMcDuff}). For these reasons, symplectic divisors have become a basic object in the study of symplectic $4$-manifolds, and more intricate configurations have subsequently attracted considerable attention.

In the absolute case, the study of symplectic $4$--manifolds is organized by the symplectic Kodaira dimension, determined by the numerical invariants $K_\omega\cdot[\omega]$ and $K_\omega^2$ (\cite{MSsurvey,Likod1,Likod2}). Here, $K_\omega=-c_1(\omega)\in H^2(X;\mathbb{Z})$ denotes the symplectic canonical class and $[\omega]\in H^2(X;\mathbb{R})$ is the cohomology class of $\omega$. This invariant provides a coarse classification that guides the study within each class. It is therefore natural to seek an analogous framework in the presence of divisors. The first attempt towards defining symplectic log Kodaira dimension was made in \cite{Lizhanglogkod} by replacing $K_\omega$ in the absolute setting by the \emph{adjoint class} $K_\omega+[D]$:
\begin{equation}\label{equ:kod}\tag{\textdagger}
\kappa^s(X,\omega,D) = 
\begin{cases}
   -\infty & \text{if }[\omega]\cdot (K_\omega+[D])<0\text{ or }(K_\omega+[D])^2<0,\\
   0 & \text{if }[\omega]\cdot (K_\omega+[D])=0\text{ and }(K_\omega+[D])^2=0,\\
   1 & \text{if }[\omega]\cdot (K_\omega+[D])>0\text{ and }(K_\omega+[D])^2=0,\\
   2 & \text{if }[\omega]\cdot (K_\omega+[D])>0\text{ and }(K_\omega+[D])^2>0,
\end{cases}
\end{equation}
where $D$ is assumed to be a disjoint union of embedded symplectic surfaces with no spherical components and to be ``maximal'', in the sense that the complement $X\setminus D$ contains no exceptional spheres. When $D$ is allowed to have spherical components and intersecting components, following the philosophy of (\ref{equ:kod}), the cases $\kappa^s=0$ and $\kappa^s=-\infty$ were later studied in \cite{Limak,LMN22} and \cite{LN25}, respectively. In \cite{LN25}, we imposed the additional \emph{connectedness} assumption on the divisor, namely that the configuration $\cup C_i$ is connected as a subset of $X$. Our primary goal in this paper is to continue the study of symplectic divisors with $\kappa^s=-\infty$ which are possibly \emph{disconnected}, thereby taking step toward a more complete picture of classification. This focus is natural in view of the fact that symplectic $4$-manifolds within the class Kodaira dimension $-\infty$ are comparatively better understood than other classes in the absolute setting (\cite[Section 4.1]{Likod2}). Moreover, (\ref{equ:kod}) suggests that it is reasonable to begin with divisors of bounded symplectic area, namely those satisfying $[\omega]\cdot [D]\leq -[\omega]\cdot K_\omega$, before turning to the subtler case in which the symplectic area of $D$ is unrestricted.

\subsection{Symplectic weighted projective planes}

Fix pairwise coprime positive integers $(a,b,c)$.
The (complex) weighted projective plane is the projective  orbifold 
\[
 \CP(a,b,c)=\bigl(\C^{3}\setminus\{0\}\bigr)\big/\C^{*},
 \qquad
 \lambda\cdot(z_{0},z_{1},z_{2})=(\lambda^{a}z_{0},\lambda^{b}z_{1},\lambda^{c}z_{2}).
\]
From the symplectic viewpoint, we will think of $\CP(a,b,c)$ as a
symplectic toric orbifold obtained by symplectic reduction of $\C^{3}$ with its
standard symplectic form.
Let $\omega_{0}=\sum_{j=0}^{2} dx_{j}\wedge dy_{j}$ be the standard symplectic form on $\C^{3}$.
Consider the Hamiltonian $S^{1}$-action with weights $(a,b,c)$:
\[
 e^{it}\cdot(z_{0},z_{1},z_{2})=(e^{iat}z_{0},e^{ibt}z_{1},e^{ict}z_{2}).
\]
One choice of moment map is
\[
 \mu(z)=\frac12\bigl(a|z_{0}|^{2}+b|z_{1}|^{2}+c|z_{2}|^{2}\bigr)\in\R.
\]
For $\tau>0$, the reduced space
\(
 \mu^{-1}(\tau)\big/S^{1}
\)
is a symplectic orbifold which is naturally identified
with $\WCP(a,b,c)$ equipped with a symplectic form $\omega_{\tau}$. Different choices of $\tau$ change the symplectic form only by overall scaling. 
The coordinate points $[1:0:0]$, $[0:1:0]$ and $[0:0:1]$
have cyclic stabilizer groups of orders $a$, $b$, and $c$, respectively.
When $(a,b,c)$ are pairwise coprime, we can choose integers
$0<a_b,a_c<a$ such that
\[
a_b b\equiv c\pmod a,
\qquad
a_c c\equiv b\pmod a.
\]
 The cyclic quotient singularities at $[1:0:0]$  are said to be of type $\frac{1}{a}(1,a_b)$ or $\frac{1}{a}(1,a_c)$. By permutation of $(a,b,c)$, we similarly define
$b_a,b_c$ modulo $b$, $c_a,c_b$ modulo $c$ and the cyclic quotient singularity types at $[0:1:0]$, and $[0:0:1]$.

The residual $T^{2}$-action on the symplectic reduction makes $\WCP(a,b,c)$ a symplectic toric orbifold.
Its moment polytope is the labeled triangle
given by the subset of $(x,y,z)\in\R_{\ge 0}^{3}$ satisfying
\(
 a x+b y+c z=\tau,
\)
with vertices $(\frac{\tau}{a},0,0)$, $(0,\frac{\tau}{b},0)$ and $(0,0,\frac{\tau}{c})$ labeled by $a$, $b$ and $c$, respectively. We refer to \cite{lermantolman,Godinho,NP09} for more background on symplectic toric orbifolds.

\begin{remark}
  Throughout the paper, we assume that the weights $(a,b,c)$ are pairwise coprime, so that the complement of three coordinate points is a smooth symplectic manifold. If $\gcd(a,b,c)=d>1$, then one may simply replace $(a,b,c)$ by $(\frac{a}{d},\frac{b}{d},\frac{c}{d})$. When $\gcd(a,b,c)=1$ but the weights are not pairwise coprime, then $\WCP(a,b,c)$ is still well defined as an orbifold, but its orbifold locus is no longer isolated: along each coordinate line $\{z_i=0\}$, the generic stabilizer has order equal to the gcd of the other two weights. In this case, the $2$-form induced by symplectic reduction degenerates along three coordinate lines so that we only get a symplectic manifold by further removing these coordinate lines. Nevertheless, according to \cite[1.3.1]{Dolgachev}, if we set
  \[a':=\frac{a}{\gcd(a,b)\gcd(a,c)},\quad b':=\frac{b}{\gcd(b,a)\gcd(b,c)},\quad c':=\frac{c}{\gcd(c,a)\gcd(c,b)},\]
  then $(a',b',c')$ is pairwise coprime, and $\CP(a',b',c')$ is isomorphic to $\CP(a,b,c)$ as projective varieties, though not as orbifolds, induced by the branched covering map
  \[\C^3 \setminus \{0\}\rightarrow \C^3 \setminus \{0\},\qquad (z_0,z_1,z_2)\mapsto (z_0^{\gcd(b,c)},z_1^{\gcd(a,c)},z_2^{\gcd(a,b)}).\]
  In particular, even without the pairwise coprime assumption, the complement of the three coordinate points in $\CP(a,b,c)$ always carries a canonical complex structure.
\end{remark}

Given $\CP(a,b,c)$, one can resolve its orbifold singularities by replacing a neighborhood of each non-Delzant corner of its moment triangle with a collection of consecutive edges satisfying the Delzant condition (see Section \ref{section:WP}). This produces a Delzant polygon corresponding to a smooth rational manifold. The new edges determine a symplectic divisor consisting of three chains of embedded spheres. Such a divisor naturally lies in the class of symplectic divisors with log Kodaira dimension $-\infty$, because it sits inside a toric symplectic log Calabi--Yau divisor (see Section \ref{section:divisorintro}) and therefore satisfies the symplectic area condition in (\ref{equ:kod}).

Consider the simple case of the weighted projective plane $\CP(1,1,c)$, which has a unique orbifold singularity at $[0:0:1]$ of order $c$. A neighborhood of this singularity can be replaced by a neighborhood of the zero section of $\mathcal{O}(-c)$. In the toric picture (Figure \ref{fig:cp11c-truncation}), this amounts to truncating the non-Delzant corner of the triangle to obtain a trapezoid, thereby producing a section of self-intersection $-c$ in a symplectic Hirzebruch surface $(\mathbb{P}(\mathcal{O}\oplus\mathcal{O}(c)),\Omega)$. There is a stratification in the space $\mathcal{J}_{\Omega}$ of $\Omega$-tame almost complex structures whose stratum $U_k$ is a Fr\'{e}chet submanifold in $\mathcal{J}_{\Omega}$ characterized by the existence of an embedded $J$-holomorphic section of self-intersection $-k$. By \cite[Proposition 2.1, Corollary 2.8]{AbreuMcDuff}, the stratum $U_{c}$ is always nonempty and connected once the symplectic area of the $(-c)$-section class is positive. Consequently, any two embedded symplectic spheres of self-intersection $-c$ are symplectically isotopic, by choosing a family of almost complex structures in $U_{c}$. Therefore, by symplectic isotopy extension theorem (\cite[Proposition 4]{Auroux} or \cite[Proposition 0.2]{SiebertTian}), one obtains the following uniqueness result for symplectic divisors with the same configuration as the one arising from the resolution of $\CP(1,1,c)$.
\begin{proposition}[\cite{AbreuMcDuff}]\label{prop:AbreuMcDuff}
 Any two embedded symplectic spheres in the $(-c)$-section class of $(\mathbb{P}(\mathcal{O}\oplus\mathcal{O}(c)),\Omega)$ are symplectomorphic.
\end{proposition}
 This naturally motivates the question of how such a uniqueness result extends to the resolution configuration for an arbitrary weighted projective plane $\CP(a,b,c)$.

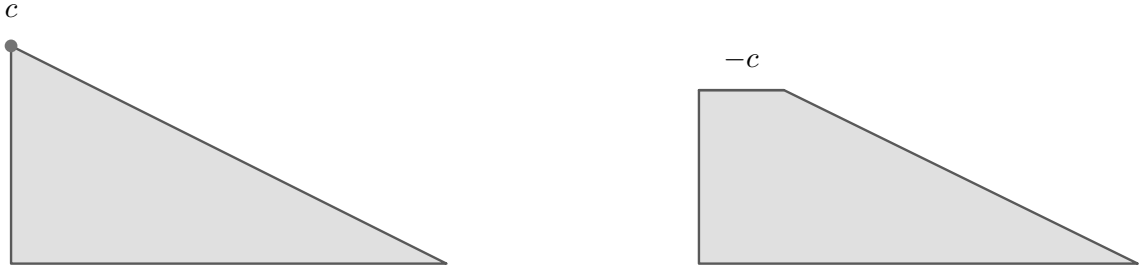
\begin{figure}[htbp]
  \centering
  \begin{tikzpicture}[x=0.9cm,y=0.9cm,line cap=round,line join=round]
    \tikzset{
      poly/.style={draw=black!65, line width=0.9pt, fill=black!12},
      pt/.style={circle, fill=black!55, inner sep=1.7pt}
    }

    \begin{scope}
      \filldraw[poly] (0,0) -- (0,3.2) -- (6.4,0) -- cycle;
      \node[pt, label={[yshift=0.15cm]above:$c$}] at (0,3.2) {};
    \end{scope}

    \begin{scope}[xshift=9.1cm]
      \filldraw[poly] (0,0) -- (0,2.55) -- (1.25,2.55) -- (6.45,0) -- cycle;
      \draw[black!65, line width=0.9pt] (0,2.55) -- (1.25,2.55);
      \node at (0.62,3.0) {$-c$};
    \end{scope}
  \end{tikzpicture}
  \caption{In the toric picture, the moment triangle of $\CP(1,1,c)$ is truncated at the non-Delzant corner, labeled by its  singularity weight $c$, to produce the trapezoid corresponding to the Hirzebruch surface $\mathbb{P}(\mathcal{O}\oplus\mathcal{O}(c))$, and the new top edge represents the section of self-intersection $-c$.}
  \label{fig:cp11c-truncation}
\end{figure}

\subsection{Main Results}

In this paper, we focus on symplectic divisors that are disjoint unions of three strings.

\begin{definition}
  A \emph{symplectic string} is a connected symplectic divisor
  \(
    \mathcal{S}
  \)
  whose irreducible components are symplectic spheres arranged in a linear chain. An \emph{orientation} $\mathfrak{o}$ of $\st$ is a labeling of its components, written as $\st=\bigcup_{i=1}^{k} S_i$, such that $[S_i]\cdot[S_j]=1$ when $|i-j|=1$ and $[S_i]\cdot[S_j]=0$ otherwise. A \emph{symplectic Hirzebruch--Jung string} (HJ string) is a symplectic string with
  \(
    [S_i]^2\le -2 \text{ for all } i.
  \) 
\end{definition}

Given a string $\st=\bigcup_{i=1}^k S_i$, we call $S_1$ and $S_k$ its \emph{endpoints}, denote its length by $l(\st):=k$, and write $[\st]:=\sum_{i=1}^k [S_i]$ for its homology class. More generally, we will also use $l(D)$ and $[D]$ to denote the number of components and total homology class for any divisor $D$. Once an orientation $\mathfrak{o}$ of the string is chosen, we can assign a rational number by taking the (negative) continued fraction
\[
  r(\st,\sto):=[b_1,\dots,b_k]=b_1-\frac{1}{b_2-\frac{1}{\ddots-\frac{1}{b_k}}}>1,
\]
where $b_i=-[S_i]^2\geq 2$. The sequence $(b_1,\cdots,b_k)$ will be called the \emph{negative self-intersection sequence} of the string $\st$. If we write $r(\st,\sto)$ as $\frac{p}{q}$ for coprime positive integers $p>q$, then the assigned continued fraction when reversing the orientation will be given by $r(\st,-\sto)=\frac{p}{q'}$ for $0<q'<p$ with $qq'\equiv 1\pmod p$.

\begin{remark}
 When $r(\st,\sto)=\frac{p^2}{pq-1}$, the HJ string $\st$ is usually called a T-string. Such a string arises as the resolution of a Wahl singularity, namely a cyclic quotient singularity of type $\frac{1}{p^2}(pq-1,1)$, which appears naturally in $\mathbb{Q}$-Gorenstein degenerations. For the weighted projective plane $\CP(a^2,b^2,c^2)$, where $(a,b,c)$ satisfies the Markov equation $a^2+b^2+c^2=3abc$, the minimal resolution consists of three T-strings. Evans--Smith (\cite{evanssmithwahl}) studied T-strings from a symplectic perspective in symplectic manifolds with $b_2^+>1$ and Kodaira dimension $2$. Using Seiberg--Witten theory for manifolds with $b_2^+>1$, they obtained a beautiful upper bound on the length of a T-string in terms of the canonical class, improving the previous bound from algebraic geometry. See also \cite{evanssmithpinwheel} for related connections with Lagrangian pinwheel embeddings. In this sense, our paper can be viewed as the $b_2^+=1$ and $\kappa^s=-\infty$ counterpart to the work of Evans--Smith: we apply Seiberg--Witten theory in the $b_2^+=1$ setting and treat the more general class of HJ strings.
\end{remark}

Let $(X,\omega)$ be a closed symplectic $4$-manifold. We use $\mathcal{E}_\omega$ to denote the set of its symplectic exceptional classes. Let $D\subseteq (X,\omega)$ be a possibly disconnected symplectic divisor with connected components $D_1,\cdots,D_n$. We introduce the set of \emph{$D$-log exceptional classes} to be 
\[\mathcal{E}^{\text{log}}_{\omega}(D):=\{E\in\mathcal{E}_\omega\,|\,E\cdot[D_\alpha]\geq 0\text{ for any }1\leq \alpha\leq n\},\]where $[D_\alpha]$ denotes the total homology class of the connected component $D_\alpha$. We do \emph{not} require $E\in \mathcal{E}_\omega^{\text{log}}(D)$ to pair non-negatively with each irreducible component of $D_\alpha$. For any two connected components $D_i,D_j$, consider the subset of their \emph{connecting log exceptional classes} defined as
\[\mathcal{E}^{\text{log}}_\omega(D_i,D_j;D):=\{E\in\mathcal{E}^{\text{log}}_{\omega}(D)\,|\, E\cdot[D_i]\geq1, E\cdot[D_j]\geq 1\}.\]
 We introduce the notion of
\emph{exceptional gap}\footnote{Conceptually, the exceptional gap can be thought as a ``naive" analogue of certain SFT-type spectral numbers defined on the divisor complement using asymptotically cylindrical punctured pseudoholomorphic curves (\cite{MScapacity}).} between $D_i$ and $D_j$ by 
\[
\mathfrak{e}_\omega(D_i,D_j;D):=
\begin{cases}
\sup\{\omega(E)\,|\,E\in\mathcal{E}^{\text{log}}_\omega(D_i,D_j;D)\}, & \mathcal{E}^{\text{log}}_\omega(D_i,D_j;D)\neq\varnothing,\\
0, & \mathcal{E}^{\text{log}}_\omega(D_i,D_j;D)=\varnothing.
\end{cases}
\]
When the symplectic form is clear from context, we omit $\omega$ and write $\mathcal{E}^{\text{log}}(D),\mathcal{E}^{\text{log}}(D_i,D_j;D)$ and $\mathfrak{e}(D_i,D_j;D)$. 

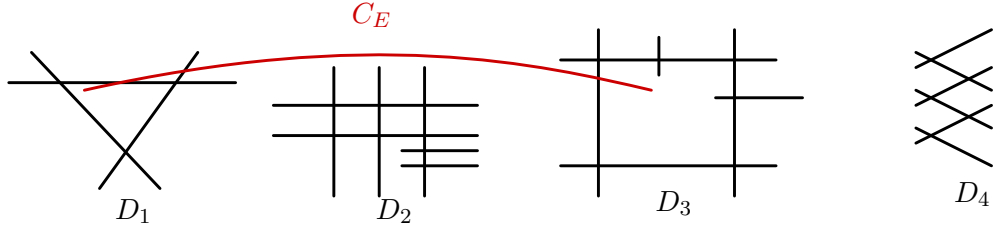
\begin{figure}[htbp]
  \centering
  \begin{tikzpicture}[x=1cm,y=1cm, line cap=round, line join=round]
    \tikzset{
      dcomp/.style={draw=black, line width=1.1pt},
      elog/.style={draw=red!80!black, line width=1.1pt}
    }

    \coordinate (Aint) at (1.35,-0.02);
    \coordinate (Aintb) at (1.00,0.31);
    \draw[dcomp] (0,0.3) -- (3,0.3);
    \draw[dcomp] (0.3,0.7) -- (2,-1.1);
    \draw[dcomp] (1.2,-1.1) -- (2.5,0.7);

    \coordinate (Bint) at (4.24,0.03);
    \coordinate (Bintb) at (4.93,-0.36);
    \coordinate (Bintc) at (5.60,-0.36);
    \coordinate (Bintd) at (4.91,0.03);
    \draw[dcomp] (4.30,-1.2) -- (4.30,0.5);
    \draw[dcomp] (4.9,-1.2) -- (4.90,0.5);
    \draw[dcomp] (5.5,-1.2) -- (5.5,0.5);
    \draw[dcomp] (3.5,0) -- (6.2,0);
    \draw[dcomp] (3.5,-0.4) -- (6.2,-0.4);
    \draw[dcomp] (5.2,-0.6) -- (6.2,-0.6);
    \draw[dcomp] (5.2,-0.8) -- (6.2,-0.8);

    \coordinate (Cint) at (7.74,0.55);
    \coordinate (Cintb) at (9.67,0.57);
    \coordinate (Cintc) at (9.66,-0.29);
    \coordinate (Cintd) at (8.58,0.56);
    \coordinate (Cinte) at (9.66,-0.08);
    \draw[dcomp] (7.8,-1.2) -- (7.8,1);
    \draw[dcomp] (9.6,-1.2) -- (9.6,1);
    \draw[dcomp] (8.6,0.9) -- (8.6,0.4);
    \draw[dcomp] (7.30,0.6) -- (10.15,0.6);
    \draw[dcomp] (7.30,-0.8) -- (10.15,-0.8);
    \draw[dcomp] (9.35,0.1) -- (10.5,0.1);

    \coordinate (Dint) at (12.53,0.10);
    \draw[dcomp] (12,0.5) -- (13,1);
    \draw[dcomp] (12,0) -- (13,0.5);
    \draw[dcomp] (12,-0.5) -- (13,0);
    \draw[dcomp] (12,0.7) -- (13,0.2);
    \draw[dcomp] (12,0.2) -- (13,-0.3);
    \draw[dcomp] (12,-0.3) -- (13,-0.8);

    \node[below] at (1.65,-1.1) {$D_1$};
    \node[below] at (5.1,-1.1) {$D_2$};
    \node[below] at (8.8,-1) {$D_3$};
    \node[below] at (12.75,-0.88) {$D_4$};

    \draw[elog]
       (1,0.20)
      .. controls (4.8,1.05) and (7.0,0.55) .. (8.50,0.20);

    
    \node[red!80!black, above] at (4.8,0.9) {$C_E$};
  
  \end{tikzpicture}
  \caption{A schematic curve configuration: the divisor $D=D_1\cup D_2\cup D_3\cup D_4$ (black) and a connecting $D$-log exceptional curve $C_E$ (red) intersecting exactly the two components $D_1$ and $D_3$. Here $E\in\mathcal{E}_{\omega}^{\mathrm{log}}(D_1,D_3;D)$, and the exceptional gap $\e_\omega(D_1,D_3;D)$ is the supremum of the symplectic areas of all such classes.}
  \label{fig:log-exceptional-configuration}
\end{figure}

\begin{definition}\label{def:3assumptions}
The disjoint union of three symplectic HJ strings $D=\st_1\cup\st_2\cup\st_3\subseteq (X,\omega)$ is called
   \begin{itemize}[nosep]
  \item \emph{full}, if $l(\st_1)+l(\st_2)+l(\st_3)=b_2^-(X)$;
  \item \emph{of $(a,b,c)$-type}, if there are pairwise coprime $(a,b,c)$ such that $r(\st_1,\sto_1)=\frac{a}{a_b}$, $r(\st_2,\sto_2)=\frac{b}{b_c}$, $r(\st_3,\sto_3)=\frac{c}{c_a}$ for some orientations $\sto_1,\sto_2,\sto_3$ on $\st_1,\st_2,\st_3$;
    \item \emph{gap-admissible}, if $[\omega]\cdot (K_\omega+[D])<-\e(\st_i,\st_j;D)-\e(\st_i,\st_k;D)$ for any $\{i,j,k\}=\{1,2,3\}$; 
    \item \emph{sub-toric}, if there exists a toric structure on $(X,\omega)$ with moment map $\mu\colon X\to \mathcal{P}$ such that $D\subseteq \mu^{-1}(\partial\mathcal{P})$.
  \end{itemize}
\end{definition}

The gap-admissible condition implies, according to (\ref{equ:kod}), that $D$ lies within the class $\kappa^s=-\infty$. The following is our main result.

\begin{theorem}\label{thm:main}
  Let $D=\st_1\cup\st_2\cup\st_3\subseteq (X,\omega)$ be a symplectic divisor consisting of three disjoint HJ strings which is full and of $(a,b,c)$-type. Then $D$ is sub-toric if and only if $D$ is gap-admissible. In particular, $D$ must arise from the symplectic minimal resolutions of weighted projective plane $\CP(a,b,c)$ under the gap-admissible assumption.
\end{theorem}
We refer to Section \ref{section:WP} for the notion of minimal resolution in the symplectic sense. Moreover, by Remark \ref{rmk:chen}, there is no need to specify the symplectic structure on $\CP(a,b,c)$ up to overall rescaling. Thus, this theorem characterizes those disconnected symplectic divisors with symplectic log Kodaira dimension $-\infty$ that arise as minimal resolutions of weighted projective planes.
Since every symplectic toric manifold admits a compatible K\"ahler structure, Theorem \ref{thm:main} also immediately implies that any divisor satisfying its hypotheses is K\"ahler; namely, there exists an $\omega$-compatible integrable complex structure $J$ on $X$ for which $D$ is a complex divisor in $(X,J)$. It is also worth emphasizing that the characterization above is entirely (co)homological.
It turns out that these symplectic divisors satisfy a Torelli-type statement: the associated pairs are determined up to symplectomorphism by their (co)homological data. The following can be viewed as a generalization of Proposition \ref{prop:AbreuMcDuff}.

\begin{corollary}[Torelli Theorem]\label{cor:torelli}
  Let $D=\st_1\cup\st_2\cup\st_3\subseteq (X,\omega)$ and $D'=\st'_1\cup\st'_2\cup\st'_3\subseteq (X',\omega')$ be symplectic divisors, each consisting of three disjoint HJ strings that are both full and of $(a,b,c)$-type. Suppose that $[\omega']\cdot K_{\omega'}<0$ and $D$ is sub-toric. Then
  \begin{itemize}[nosep]
    \item if there is an integral isometry $\gamma$ between the second (co)homologies of $X$ and $X'$ that sends $[\omega]$ to $[\omega']$ and $[\st_i]$ to $[\st_i']$ for each $i$, then $D'$ is also sub-toric;
    \item if, in addition, $\gamma$ preserves the class of each component of $D$ and $D'$, then $(X',\omega',D')$ is symplectomorphic to $(X,\omega,D)$.
  \end{itemize}
\end{corollary}

This corollary will be deduced from the Torelli theorem (Theorem \ref{thm:LCYtorelli}) for symplectic log Calabi--Yau pairs, introduced in Section \ref{section:divisorintro}. The novelty of the first item in Corollary \ref{cor:torelli} is that it requires only the total homology classes of the three HJ strings to agree, rather than the homology classes of their individual irreducible components as the second item.



Other than imposing the area conditions on $D$ discussed above, one can characterize the minimal resolutions of symplectic weighted projective planes in terms of the existence of symplectic rulings. On the algebraic side, a theorem of Daigle--Russell \cite{DP2} (see also \cite{DP1}) asserts that if a complete normal affine-ruled rational surface $X$ has smooth locus $X_s$ of Picard rank one and has the same resolution graph as $\CP(a,b,c)$, then $X$ is isomorphic to $\CP(a,b,c)$. Here, affine-ruledness means that $X_s$ contains a Zariski-open dense subset isomorphic to the product of a curve and the affine line. On the symplectic side, we defined and verified the symplectic version of affine-ruledness in our previous work \cite{LN25} for divisor complement of any connected symplectic divisor with $[\omega]\cdot(K_\omega+[D])<0$ by searching for suitable singular symplectic curves foliating the divisor complement. Recall that $C\subseteq (X,\omega)$ is a \emph{symplectic $(p,q)$-unicuspidal curve} if there exists a unique singular point $x\in C$ locally modeled on the $(p,q)$-cusp singularity of the algebraic curve $\{x^p+y^q=0\}\subseteq \C^2$, which is otherwise an embedded symplectic submanifold (\cite[Definition 3.5.1]{McduffSiegel}). It is called \emph{rational} if it admits a parametrization by a smooth sphere. Given a symplectic divisor $D$ with two intersecting components $D_p$ and $D_q$, we say that $C$ \emph{has a $(p,q)$-cusp} at $D_p\cap D_q$ if its unique singular point is $x=D_p\cap D_q$ and if $C$ has contact orders $p$ and $q$ with $D_p$ and $D_q$, respectively (\cite[Section 2.2]{MScounting}). The following theorem may be viewed as a symplectic analogue of the Daigle--Russell's result.

\begin{theorem}\label{thm:ruling}
Let $D=\st_1\cup\st_2\cup\st_3$ be a symplectic divisor consisting of three disjoint HJ strings in a symplectic $4$-manifold $(X,\omega)$. Suppose $D$ is full and of $(a,b,c)$-type. Then $D$ is sub-toric if and only if there exists $C\subseteq (X,\omega)$ which is
\begin{itemize}[nosep]
  \item either a symplectic $(p,q)$-unicuspidal rational curve with $[C]^2=pq$, whose $(p,q)$-cusp is at the intersection of two components of one string $\st_i$ and disjoint from other components in $D$;
  \item or an embedded symplectic sphere of self-intersection $0$ intersecting exactly one component in $D$ transversely.
\end{itemize} 
\end{theorem}

\begin{remark}
  From the symplectic point of view, the theorem above may also be regarded as a ``weighted'' version of the classical Gromov--McDuff's theorem characterizing the symplectic projective plane by the existence of an embedded symplectic sphere of self-intersection $1$. Note that we do not use the notion of orbifold curves, but we do need to allow for singular curves on smooth manifolds.
\end{remark}

\begin{remark}
  The symplectic area of the class $[C]$ in Theorem \ref{thm:ruling} is expected to provide an efficient upper bound for the Gromov width of the smooth locus of $\CP(a,b,c)$. Combined with the enumeration result of \cite{DP2}, this suggests investigating the infimum of the symplectic areas of all affine ruling classes and comparing it with the actual Gromov width, or with the lattice width of the moment polygon of $\CP(a,b,c)$.
\end{remark}

\subsection{Structure of the paper}

In Section \ref{section:divisorintro}, we introduce the basic notions and background on symplectic divisors used throughout the paper. Since our focus is on symplectic (log) Kodaira dimension $-\infty$, the notion of a fiber class plays an important role. Subsections \ref{section:degfiberclass} and \ref{section:resfiberclass} investigate the behavior of degenerations of fiber classes and explain how to obtain them from non-negative definite strings, respectively. In Subsection \ref{sec:LN25}, we then review the main technique from our previous paper \cite{LN25} for dealing with connected symplectic divisors with $\kappa^s=-\infty$.

\medskip
\noindent In Section \ref{section:WP}, we study weighted projective planes and their minimal resolutions in the symplectic sense. There we prove the ``only if'' direction of Theorem \ref{thm:main}.

\medskip
\noindent In Section \ref{Section:HJstrings}, we prove the ``if'' direction. Subsection \ref{section:HJnontoric} investigates the restrictions imposed by the full, $(a,b,c)$-type, and gap-admissible conditions in Definition \ref{def:3assumptions} on the intersection patterns between exceptional spheres and HJ strings. In Subsections \ref{section:connecting} and \ref{section:complete}, we show how to connect three strings into a single string and then embed it into some toric boundary divisor.

\medskip
\noindent In Section \ref{section:ruling}, we study the interaction between the existence of symplectic affine rulings and the sub-toric configurations arising from minimal resolutions of weighted projective planes. There we prove Theorem \ref{thm:ruling} by adapting the strategies developed in the previous sections.

\medskip
\noindent Throughout the paper, we will sometimes write $C_A$ for a symplectic submanifold of homology class $[C_A]=A$, so as to make a distinction between configurations and their classes.

\medskip
\noindent{\bf Acknowledgement:}  The authors  are supported by Simons Travel Fund SFI-MPS-TSM-00013390.

\section{Preliminaries on symplectic divisors}\label{section:divisorintro}

In this section, we review several birational operations on symplectic divisors. Let $D$ be a symplectic divisor in $(X,\omega)$.
 A symplectic embedding $I$ of the standard ball $B(\delta)\subseteq (\C^2,\omega_{\mathrm{std}})$ into $(X,\omega)$ is said to be \emph{relative to $D$} if $I^{-1}(D)$ is either empty or the union of one or two coordinate planes in $\C^2$, intersected with $B(\delta)$. Since $D$ is $\omega$-orthogonal, for sufficiently small $\delta$ there exists a $D$-relative symplectic embedding around every point of $X$. The symplectic blowup construction (\cite[Section 7.1]{MSbook}) removes $I(B(\delta))$ and collapses its boundary along the Hopf fibration to produce a symplectic exceptional sphere $C_E$ of symplectic area $\delta$. When the embedding is relative to the divisor $D$, the intersection of $D$ with the boundary of $I(B(\delta))$ consists precisely of circle fibers of the Hopf fibration. The \emph{proper transform} $\tilde{D}_i$ of a component $D_i\subseteq D$ is then the symplectic submanifold in the blowup manifold $(\tilde{X},\tilde{\omega})$ obtained by removing the interior of the disk $I(B(\delta))\cap D_i$ and collapsing its boundary circle. The proper transform $\tilde{D}_i$ intersects the exceptional sphere $C_E$ $\tilde{\omega}$-orthogonally. Therefore, the union of all proper transforms $\tilde{D}_i$ together with $C_E$ is a symplectic divisor in the blowup symplectic manifold $(\tilde{X},\tilde{\omega})$, called the \emph{total transform} of $D$.

Depending on the position of the symplectic embedding $I$ relative to the divisor $D$, we say that $(\tilde{X},\tilde{\omega},\tilde{D})$ is obtained from $(X,\omega,D)$ by an \emph{exterior blowup} if $I(B(\delta))$ is disjoint from $D$ and $\tilde{D}$ is either the total transform of $D$ or the union of all proper transforms $\tilde{D}_i$; by a \emph{toric blowup} if $I(B(\delta))$ is centered at an intersection point of $D$ and $\tilde{D}$ is the total transform of $D$; by a \emph{non-toric blowup} if $I(B(\delta))$ intersects only one component of $D$ and $\tilde{D}$ is the union of all proper transforms $\tilde{D}_i$; and by a \emph{half-toric blowup} if $I(B(\delta))$ intersects only one component of $D$ and $\tilde{D}$ is the total transform of $D$.
Accordingly, the symplectic exceptional sphere $C_E$ is called exterior, toric, non-toric, or half-toric with respect to the pair $(D,\tilde{D})$, and the inverse operation is called the corresponding blowdown.

\begin{figure}[t]
\centering
\begin{tikzpicture}[x=0.99cm,y=0.99cm, line cap=round, line join=round, >=Latex,
    divisor/.style={line width=1.1pt, draw=black},
    branch/.style={line width=1.1pt, draw=black},
    ghost/.style={line width=1.1pt, draw=black},
    newcurve/.style={line width=1.1pt, draw=blue!70},
    exc/.style={line width=0.8pt, draw=black!65},
    labelstyle/.style={font=\bfseries\Large}]

  \draw[divisor] (6.4,2.35) .. controls (7.6,2.7) and (9.1,2.7) .. (10.4,2.25);
  \draw[ghost] (7.1,1.7) .. controls (7.15,2.35) and (6.95,3.15) .. (6.75,3.6);
  \draw[branch] (7.55,1.7) .. controls (7.7,2.3) and (7.65,3.1) .. (7.4,3.8);
 
  \draw[branch] (9.1,1.5) .. controls (9.15,2.2) and (9.1,3.0) .. (9.45,3.95);
   \draw[branch] (8.8,3.5) .. controls (9.15,3.5) and (9.1,3.4) .. (10.15,3.55);
  
  \draw[exc] (7,2.95) circle (0.34);
  \draw[exc] (8.38,3.63) circle (0.32);
  \draw[exc] (9.12,2.52) circle (0.3);
  \node at (8.1,2.08) {$D$};
  \node[font=\small] at (6.08,2.92) {$I(B(\delta))$};
  \node[font=\small] at (8.43,4.16) {$I(B(\delta))$};
  \node[font=\small] at (10.05,2.7) {$I(B(\delta))$};

  \draw[divisor] (0.9,4.55) .. controls (2.1,4.9) and (3.5,4.9) .. (4.75,4.48);
  \draw[ghost] (1.45,4.0) .. controls (1.55,4.6) and (1.45,5.35) .. (1.25,5.85);
  \draw[branch] (1.95,4.0) .. controls (2.1,4.55) and (2.05,5.3) .. (1.82,6.02);
  \draw[branch] (4.08,3.88) .. controls (4.15,4.58) and (4.1,5.28) .. (4.48,6.12);
    \draw[branch] (3.8,5.8) .. controls (4.45,5.7) and (4.45,5.7) .. (5.2,5.8);
  \draw[newcurve] (0.8,5.38) .. controls (1.25,5.3) and (1.6,5.3) .. (1.7,5.46);
  \node at (2.95,4.18) {$\tilde{D}$};

  \draw[divisor] (0.85,0.82) .. controls (2.0,1.12) and (3.45,1.12) .. (4.7,0.78);
  \draw[ghost] (1.48,0.36) .. controls (1.56,0.96) and (1.46,1.72) .. (1.22,2.2);
  \draw[branch] (1.96,0.35) .. controls (2.1,0.92) and (2.06,1.7) .. (1.84,2.5);
 
  \draw[branch] (4.08,0.3) .. controls (4.12,0.95) and (4.08,1.75) .. (4.5,2.42);
  \draw[branch] (3.8,1.8) .. controls (4.45,1.7) and (4.45,1.7) .. (5.2,1.8);
 
  \draw[newcurve,dashed] (0.8,1.58) .. controls (1.25,1.52) and (1.6,1.55) .. (1.7,1.72);
  \node at (2.95,0.42) {$\tilde{D}$};

  \draw[divisor] (12.9,4.95) .. controls (14.25,5.35) and (15.65,5.35) .. (16.95,4.85);
  \draw[ghost] (13.35,4.55) .. controls (13.45,5.15) and (13.42,5.95) .. (13.2,6.55);
  \draw[branch] (13.82,4.52) .. controls (13.98,5.15) and (13.97,5.98) .. (13.75,6.75);
  
  \draw[branch] (16.35,4.38) .. controls (16.35,5.05) and (16.3,5.85) .. (16.85,6.85);
   \draw[branch] (16.25,6.4) .. controls (16.35,6.4) and (16.4,6.3) .. (17.35,6.6);
  
  \draw[newcurve] (14.55,6.62) .. controls (15.05,6.15) and (15.62,5.95) .. (16.02,5.98);
  \draw[newcurve,dashed] (14.38,7.15) .. controls (14.95,6.86) and (15.4,6.55) .. (16.1,6.52);
  \node at (15.0,4.53) {$\tilde{D}$};
  \node at (14.18,6.78) {or};

  \draw[divisor] (13.25,1.12) .. controls (14.45,1.45) and (15.52,1.45) .. (16.75,0.95);
  \draw[ghost] (13.78,0.72) .. controls (13.88,1.25) and (13.82,1.95) .. (13.6,2.45);
  \draw[branch] (14.25,0.7) .. controls (14.38,1.26) and (14.35,2.0) .. (14.14,2.6);
  \draw[branch] (16.88,1.45) .. controls (16.74,1.45) and (16.7,2.02) .. (17.22,3.12);
  \draw[ghost] (16.64,2.78) .. controls (16.88,2.55) and (17.02,2.42) .. (17.38,2.42);
  \draw[newcurve] (16.1,0.8) .. controls (16.18,1.55) and (16.42,2.0) .. (16.95,2.05);
  \node at (15.0,0.7) {$\tilde{D}$};

  \draw[->, draw=black!65, line width=0.6pt] (6.5,3.65) -- (4.85,5.35);
  \node[labelstyle] at (5.95,5.0) {half-toric};

  \draw[->, draw=black!65, line width=0.6pt] (6.25,2.55) -- (4.85,1.45);
  \node[labelstyle] at (5.55,2.25) {non-toric};

  \draw[->, draw=black!65, line width=0.6pt] (8.95,4.05) -- (12.55,6.15);
  \node[labelstyle] at (10.7,5.55) {exterior};

  \draw[->, draw=black!65, line width=0.6pt] (9.45,2.2) -- (12.95,1.25);
  \node[labelstyle] at (11.6,1.98) {toric};

\end{tikzpicture}
\caption{Four types of blowups, where the dashed curve indicates a component not contained in $\tilde{D}$. To distinguish the types, one must consider both $(X,\omega,D)$ and $(\tilde{X},\tilde{\omega},\tilde{D})$. For example, comparing $(X,\omega)$ with $(\tilde{X},\tilde{\omega})$ determines the exceptional class, while comparing $D$ with $\tilde{D}$ distinguishes the half-toric and non-toric cases.}
\label{fig:blowup-types}
\end{figure}
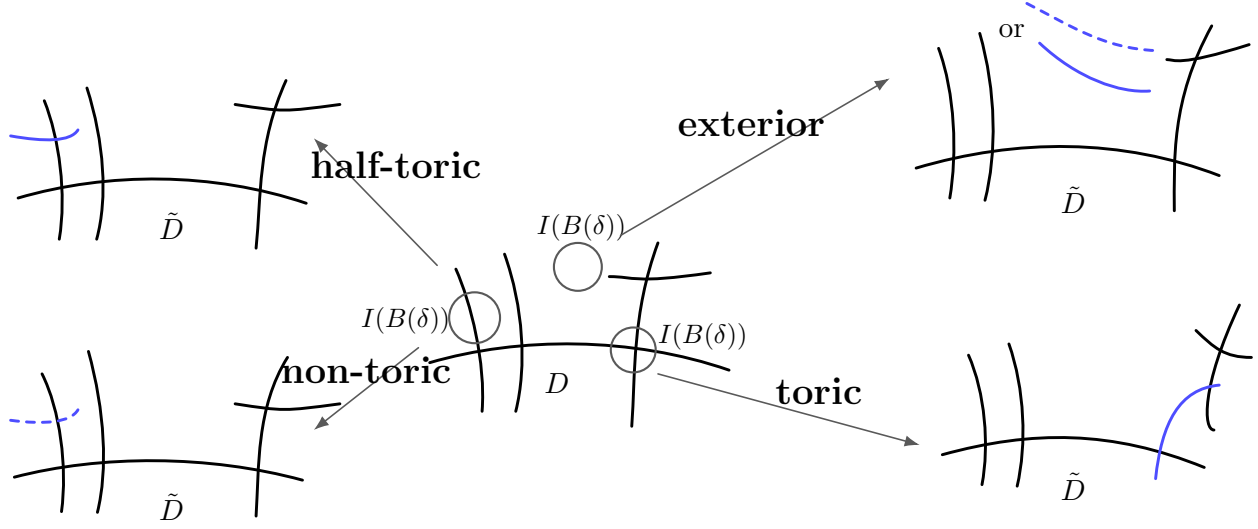

Recall that given a pair $(X,\omega,D)$, its adjoint class is defined to be $K_\omega+[D]$. We have the following simple observation.

\begin{lemma}\label{lem:blowdownarea}
Under any blowdown (toric, non-toric, half-toric, or exterior), the symplectic area of the adjoint class is non-increasing.
\end{lemma}
\begin{proof}
 Let $(X,\omega,D)$ be the blowdown of $(\tilde{X},\tilde{\omega},\tilde{D})$ along the exceptional sphere $C_E$. Then $K_{\tilde{\omega}}-K_{\omega}=E$, while $[\tilde{D}]-[D]$ is either $-E$ or $0$ depending on the type of blowdown. Therefore, the symplectic area of the adjoint class either decreases by $\omega(E)>0$ or remains unchanged under the blowdown.
\end{proof}

Following \cite{MO15}, we introduce the almost complex structures which we will work with for the symplectic divisor $D=\cup D_i\subseteq (X,\omega)$.
\begin{definition}
    An $\omega$-tame almost complex structure $J$ is said to be \emph{$D$-adpated} if there exists some plumbed closed fibered neighborhood $\overline{\mathcal{N}}(D)=\cup\overline{\mathcal{N}}(D_i)$, where each $\overline{\mathcal{N}}(D_i)$ is a closed neighborhood of $D_i$ modeled on a closed neighborhood of the zero section in a holomorphic line bundle over $D_i$ with Chern number $[D_i]^2$, such that $J$ is integrable in $\overline{\mathcal{N}}(D)$ and makes both $D_i$ and the projection $\overline{\mathcal{N}}(D_i)\rightarrow D_i$ be $J$-holomorphic. The collection of all such $J$'s will be denoted by $\mathcal{J}(D)$.
\end{definition}

\begin{definition}\label{def:Dgood}
    Let $D\subseteq (X,\omega)$ be a symplectic divisor. A nonzero class $A\in H_2(X;\mathbb{Z})$ is said to be \emph{$D$-good} if the following conditions hold:
    \begin{enumerate}[nosep]
        \item $SW(A)\neq 0$;
        \item if $A^2=0$, then $A$ is primitive;
        \item $A\cdot E\geq 0$ for any exceptional class $E\in\mathcal{E}_\omega\setminus \{A\}$;
        \item $A\cdot[D_i]\geq 0$ for all $i$.
    \end{enumerate}
\end{definition}
Here, the first condition $SW(A)\neq 0$ means that the Seiberg--Witten invariant of the spin$^{c}$ structure associated to $A$ is nonzero; when $b_2^+(X)=1$, this is understood to be the invariant defined within the symplectic chamber (see \cite[Section 2.2]{LN25} for more discussions), since $K_\omega$ determines a canonical bijection between $H_2(X;\mathbb{Z})$ and the set of spin$^c$ structures. In particular, this implies that its \emph{SW-index}
\[I_{SW}(A):=A^2-K_\omega\cdot A\geq 0.\] 

The following theorem of McDuff--Opshtein provides an effective criterion for the existence of embedded symplectic surfaces in the relative setting.

\begin{theorem}[\cite{MO15}, Theorem 1.2.7]\label{thm:MO15}
    If $A\in H_2(X;\mathbb{Z})$ is a $D$-good class such that $A^2+K_{\omega}\cdot A=-2$, then there is a residual set $\mathcal{J}_{\text{emb}}(D,A)\subseteq \mathcal{J}(D)$ such that for any $J\in \mathcal{J}_{\text{emb}}(D,A)$, $A$ can be represented by an embedded $J$-holomorphic sphere. 
\end{theorem}

The lemma below is an immediate consequence of McDuff--Opshtein's criterion.

\begin{lemma}\label{lem:blowupexceptional}
  Let $(X,\omega,D)$ be a pair, and let $(\tilde{X},\tilde{\omega},\tilde{D})$ be a toric, non-toric, half-toric, or exterior blowup of $(X,\omega,D)$. Assume that $C\subseteq (X,\omega)$ is a symplectic exceptional sphere such that $C\cup D$ is also a symplectic divisor. Then there exists a symplectic exceptional sphere $\tilde{C}\subseteq (\tilde{X},\tilde{\omega})$ such that $[\tilde{C}]=[C]$ under the natural inclusion $H_2(X;\mathbb{Z})\subseteq H_2(\tilde{X};\mathbb{Z})$, and $\tilde{C}\cup \tilde{D}$ is also a symplectic divisor.
\end{lemma}
\begin{proof}
  Any symplectic exceptional class automatically satisfies conditions (1), (2), and (3) of Definition \ref{def:Dgood}. Since $C\cup D$ is a symplectic divisor, we have $[C]\cdot[D_i]\geq 0$ for every component $D_i$ of $D$. After the blowup, we still have $[C]\cdot[\tilde{D}_i]\geq 0$, because $[\tilde{D}_i]-[D_i]$ is either $0$ or $-E$, where the exceptional class $E$ satisfies $E\cdot[C]=0$. Therefore $[C]$ is $\tilde{D}$-good, and Theorem \ref{thm:MO15} yields an embedded $J$-holomorphic representative $\tilde{C}$ for some $J\in\mathcal{J}(\tilde{D})$. If necessary, we may perturb $\tilde{C}$ slightly so that each intersection between $\tilde{C}$ and $\tilde{D}$ is positively transverse and $\tilde{\omega}$-orthogonal, by \cite[Corollary 3.4]{LiUsher} and \cite[Lemma 2.3]{Gompf}, respectively. Hence $\tilde{C}\cup\tilde{D}$ is also a symplectic divisor.
\end{proof}

A symplectic divisor $D\subseteq (X,\omega)$ is called \emph{log Calabi--Yau} if $[D]=\operatorname{PD}(c_1(\omega))$. By \cite[Lemma 3.1]{Limak}, if $D$ is not a single torus, then $D$ must be a cycle of spheres and $X$ must be rational. It is known (\cite{ATF}) that a symplectic log Calabi--Yau divisor $D$ can be viewed as the boundary divisor of an almost toric fibration on $(X,\omega)$. In the extreme case where the almost toric fibration is in fact toric, induced by a Hamiltonian $T^2$-action, $D$ attains the maximal number of components,
\[
  l(D)=b_2(X)+2.
\]
Such divisors will be called \emph{toric log Calabi--Yau}. The following symplectic Torelli theorem may be viewed as a generalization of the classical Delzant theorem (\cite{Delzant}) from the toric setting to the almost toric setting.

\begin{theorem}[\cite{Limak,LMN22}]\label{thm:LCYtorelli}
  Two symplectic log Calabi--Yau pairs $(X,\omega,D=\cup D_i)$ and $(X',\omega',D'=\cup D_i')$ are symplectomorphic if and only if there is an integral isometry between the second (co)homologies of $X$ and $X'$ that sends $[\omega]$ to $[\omega']$ and $[D_i]$ to $[D_i']$ for each $i$.
\end{theorem}


\subsection{Degeneration of fiber class}\label{section:degfiberclass}

Given a closed symplectic $4$-manifold $(X,\omega)$, if $C$ is an embedded symplectic sphere of self-intersection $0$, we will call $[C]$ a \emph{fiber class}. When such a fiber class exists, McDuff's theorem (\cite{Mcduffrationalruled}) implies that $(X,\omega)$ must be either rational or ruled which in particular implies $b_2^+(X)=1$. We will need the following elementary observation about the complement of fiber classes later.

\begin{lemma}\label{lem:fibercomplement}
  Assume $[C]$ is a fiber class of $(X,\omega)$ represented by an embedded symplectic sphere $C$. Then there cannot exist a negative definite divisor configuration with $b_2^-(X)$ components in the complement  $X\setminus C$. 
\end{lemma}

\begin{proof}
  Otherwise, the homology classes of the divisor components span a negative definite sublattice in $\langle [C]\rangle^\perp\subseteq H_2(X;\mathbb{Z})$. Because $b_2^+(X)=1$, $\langle [C]\rangle^\perp$ has rank $b_2^-(X)$. Since $[C]\in \langle [C]\rangle^\perp$, $\langle [C]\rangle^\perp$ cannot contain a negative definite sublattice of rank $b_2^-(X)$.
\end{proof} 
For any $\omega$-tame almost complex structure $J$, a \emph{$J$-holomorphic subvariety $\Theta$}, in the sense of Taubes (\cite{Taubessubvariety}), is a finite collection $\{(C_i,m_i)\}$ where each $m_i$ is a positive integer and $C_i$ is a closed subset of $X$ such
\begin{itemize}[nosep]
 \item its $2$-dimensional Hausdorff measure is finite and non-zero;
        \item it has no isolated points;
        \item away from finitely many singular points, $C_i$ is a connected smooth submanifold with $J$-invariant tangent space.
\end{itemize} 
Each $C_i$ is called an irreducible component of $\Theta$, and the homology class of $\Theta$ is defined to be $\sum m_i[C_i]$. For a fiber class $[C]$, the Seiberg--Witten invariant is nontrivial, with index
\[I_{SW}([C])=[C]^2-K_\omega\cdot[C]=2[C]^2+2=2.\]
It follows that, for any point $x\in X$, there exists a $J$-holomorphic subvariety in class $[C]$ passing through $x$. Now, we investigate the possibly reducible $J$-holomorphic representatives of a fiber class when $[C]$ is $J$-nef. Related results, formulated in terms of ``broken rulings'', also appear in \cite[Section 4]{Markovstaircase}.

\begin{proposition}\label{prop:fiberclass}
  Suppose $F$ is a fiber class of $(X,\omega)$ and $J$ is an $\omega$-tame almost complex structure such that $F$ has an embedded $J$-holomorphic representative. Let $\Theta=\{(C_i,m_i)\}$ be a reducible $J$-holomorphic subvariety in class $F$. Then we have the following.
  \begin{enumerate}[nosep]
  \item The configuration $\Gamma=\cup C_i$ is a conntected tree of embedded spheres.
   \item The intersection matrix of any proper subtree $\Gamma'\subsetneq\Gamma$ is negative definite.
  \item In the configuration $\Gamma$, there must be either at least two $(-1)$-spheres, or one $(-1)$-sphere with multiplicity $m_i\geq 2$. 
\end{enumerate}
\end{proposition}

\begin{proof}
  The first statement follows from \cite[Theorem 1.5]{Jnef} since the assumption guarantees the $J$-nefness of the fiber class $F$. 

  To prove the second statement, suppose there was a nonzero class
\[
A=\sum_{C_i\subseteq \Gamma'} a_i [C_i],
\]
supported on a proper subtree $\Gamma'$, such that $A^2\ge 0$, where all $a_i\neq 0$. Since $F$ admits an embedded representative, we have $F\cdot [C_i]=0$ for all $i$, and hence $F\cdot A=0$. As $b_2^+(X)=1$, the light cone lemma (\cite[Lemma 3.7]{McDufflectures}) implies that $A$ is proportional to $F$.

Now suppose $\Gamma'$ has $m$ connected components, and write
\[
A=A_1+\cdots +A_m,
\]
where each $A_j$ is supported on a single connected component of $\Gamma'$. Since the components are disjoint, we have
\[
A^2=\sum_{j=1}^m A_j^2=0.
\]
Moreover, each $A_j$ satisfies $F\cdot A_j=0$ so that again by light cone lemma each $A_j$ is proportional to $F$, and hence $A_j^2=0$ for all $j$.

Now let $\Gamma''$ be any connected component of $\Gamma\setminus \Gamma'$, and define
\[
B:=\sum_{C_i\subseteq \Gamma''} [C_i].
\]
Since $\Gamma$ is connected and contains no loops by the first statement, there exists some component of $\Gamma'$ that meets $\Gamma''$ in exactly one point. In particular, there is some $A_j$ whose support intersects $\Gamma''$ at a unique component $C_\ell\subseteq \Gamma'$. It follows that
\[
B\cdot A_j = a_\ell \neq 0.
\]
On the other hand, since $A_j$ is proportional to $F$ and $F\cdot B=0$, we must have
\[
B\cdot A_j=0,
\]
which is a contradiction. This shows that no such class $A$ can exist, and hence every subtree $\Gamma'$ must be negative definite.

For the third statement, applying the adjunction formula to the fiber class $F=\sum m_i[C_i]$, we obtain
\[
-2 = K_\omega \cdot F = \sum m_i(-2 - [C_i]^2).
\]

We first claim that $F$ is primitive. Indeed, if $F=2F'$, then
\[
(F')^2 + K_\omega \cdot F' = -1,
\]
which contradicts the Wu formula (\cite[Proposition 1.4.18]{GSbook}), since $K_\omega$ is an integral lifting of $w_2$.

Next, we observe that no component $C_i$ can satisfy $[C_i]^2 \ge 0$. Otherwise, since $F\cdot [C_i]=0$, the light cone lemma would force $[C_i]$ to be proportional to $F$, contradicting the fact that $F$ is primitive and supported on multiple components.

Therefore, all components satisfy $[C_i]^2 \le -1$. Returning to the identity
\[
-2 = \sum m_i(-2 - [C_i]^2),
\]
we note that each term $-2 - [C_i]^2 \ge -1$, with equality if and only if $[C_i]^2=-1$. It follows that in order for the sum to equal $-2$, there must be either at least two components with $[C_i]^2=-1$, or one component with $[C_i]^2=-1$ appearing with multiplicity $m_i \ge 2$.
\end{proof}

\subsection{Building fiber classes from non-negative definite strings}\label{section:resfiberclass}
In this section, we explain how to extract a fiber class from a non-negative definite symplectic string, possibly after a sequence of toric blowups. Given an oriented string $\st=\cup_{i=1}^n S_i\subseteq (X,\omega)$ with negative self-intersection sequence $(b_1,\cdots,b_n)$, one can consider its \emph{negative intersection matrix} \[M(\st):=\begin{pmatrix}
b_1 & -1    & 0    & \cdots & 0 \\
-1    & b_2 & -1    & \cdots & 0 \\
0    & -1    & b_3 & \cdots & 0 \\
\vdots & \vdots & \vdots & \ddots & -1 \\
0    & 0    & 0    & -1      & b_n
\end{pmatrix}.\]
For $1 \le k \le n$, let $M_k(\st)$ denote the $k\times k$ leading principal submatrix of $M(\st)$, and define integers \[\Delta_1:=1,\quad \Delta_l:=\det M_{l-1}(\st), \quad 2\leq l\leq n+1,\] which will satisfy the recurrence relation \[\Delta_{l}=b_{l-1}\Delta_{l-1}-\Delta_{l-2},\quad 3\leq l\leq n+1.\]
One can see that $\gcd(\Delta_l,\Delta_{l+1})=1$ and the continued fraction $[b_l,b_{l-1},\cdots,b_1]=\frac{\Delta_{l+1}}{\Delta_l}$ by induction. Moreover, one checks directly that a toric blowdown or half-toric blowdown leaves the determinant $\det M(\st)$ of the negative intersection matrix of the entire string unchanged. If $\st$ is in addition assumed to be Hirzebruch--Jung, it then follows from the above recurrence relation that $0<\Delta_1<\Delta_2<\cdots<\Delta_{n+1}$ since all $b_i\geq 2$. In particular, this implies that $M(\st)$ is positive definite for HJ string $\st$ by Sylvester's criterion. Equivalently, the intersection matrix $-M(\st)$ (and by abuse of notation, the string $\st$ itself) is negative definite. Conversely, if $\st$ is known to be not negative definite, then there must be some $\Delta_l\leq 0$. The following facts will also be needed later.

\begin{lemma}\label{lem:-1endpoint}
  If $b_1=1$ and $b_i\geq 2$ for any $2\leq i\leq n$, then $\mathcal{S}$ is negative definite.
\end{lemma}

\begin{proof}
  Let $\mathcal{S}'$ be the string obtained by performing a half-toric blowdown along the $S_1$-sphere. Then $\mathcal{S}$ is negative definite if and only if $\mathcal{S}'$ is. The claim therefore follows by induction.
\end{proof}

\begin{lemma}\label{lem:adjacent-1}
  Assume that for some $1\leq k\leq n$, we have $b_k=1$ and $b_i\neq 1$ for $i\neq k$. Let $\st'$ be the string obtained from $\st$ by successive toric or half-toric blowdowns, with negative self-intersection sequence $(b_1',\cdots,b_m')$ where $m<n$. Then at most two of the $b_i'$ are equal to $1$, and if there are two, they must be adjacent.
\end{lemma}

\begin{proof}
  After blowing down the sphere $S_k$, the resulting string has negative self-intersection sequence $(b_1,\cdots,b_{k-1}-1,b_{k+1}-1,\cdots,b_n)$, so the only entries that can become $1$ are $b_{k-1}-1$ and $b_{k+1}-1$. This proves the claim when $m=n-1$. If we continue by blowing down either the $(b_{k-1}-1)$-component or the $(b_{k+1}-1)$-component, the resulting string still has at most one exceptional component, or else two adjacent exceptional components. The general case $m<n$ therefore follows by induction.
\end{proof}

 One can further define the homology class 
\[F_k:=\sum_{i=1}^k\Delta_i[S_i]\]
and easily check that it enjoys the intersection property 
\[F_k^2=-\Delta_k\Delta_{k+1},\qquad K_\omega\cdot F_k=-\Delta_k+\Delta_{k+1}-1,\qquad F_k\cdot [S_i]=\begin{cases}
  -\Delta_{k+1},\qquad & i=k\\
  \Delta_k,\qquad & i=k+1\\
  0,\qquad & \text{otherwise.}
\end{cases}\]

Assume now that $\Delta_k>0$ and $\Delta_{k+1}\leq 0$; then we obtain a pair of coprime non-negative integers
\(
(p,q):=(-\Delta_{k+1},\Delta_k).
\) To the pair $(p,q)$, one can associate a \emph{weight sequence} $\mathcal{W}(p,q)=(m_1,m_2,\ldots,m_L)$ of positive integers, defined as follows. If $(p,q)=(0,1)$, set $\mathcal{W}(p,q)=\varnothing$. Otherwise, set $(p_1,q_1):=(p,q)$, and whenever $(p_i,q_i)$ is defined with $p_i\neq q_i$, define
$(p_{i+1},q_{i+1}):=(|p_i-q_i|,\min\{p_i,q_i\})$. Since $p$ and $q$ are coprime, there exists $L\in\mathbb{Z}_{+}$ such that $(p_L,q_L)=(1,1)$. For each $i$, define
$m_i:=\min\{p_i,q_i\}$. If we take a rectangular box of size $p\times q$, the weight sequence corresponds to a decomposition of this rectangle into squares of side lengths $m_1,\dots,m_L$ (see \cite[Section 4.1]{McduffSiegel}). By comparing the area and half-perimeter contributions of these squares (the last square of side length $m_L=1$ contributes twice to the half-perimeter), we have
\[pq=\sum_{i=1}^L m_i^2,\qquad p+q=\sum_{i=1}^L m_i+1.\]

There is also a pattern of successive toric blowups associated to the pair $(p,q)\neq (0,1)$. First perform toric blowup at the node $S_k\cap S_{k+1}$ and take the total transform $\st'$ of $\st$. The new string $\st'$ can be denoted by $S_1'\cup\cdots\cup S_k'\cup C_1\cup S_{k+1}'\cup\cdots\cup S_n'$, where $S_i'$ is the proper transform of $S_i$ and $C_1$ is the exceptional component. Next, depending on whether $p_1>q_1$ or $p_1<q_1$, we perform toric blowup at the node $S_k'\cap C_1$ or $C_1\cap S_{k+1}'$ to obtain the string $\st''$ containing a new exceptional component $C_2$. Depending on whether $p_2>q_2$ or $p_2<q_2$, one can further choose the `left' or `right' node to perform toric blowup. Repeat this procedure until we reach at $(p_L,q_L)=(1,1)$ after $L$ blowups. The result will be a symplectic string $\tilde{\st}\subseteq(X\# L\overline{\CP}^2,\tilde{\omega})$ satisfying $l(\tilde{\st})=l(\st)+L$. Denote by $\tilde{C}_i$ the component in $\tilde{\st}$ that comes from the proper transform of the exceptional sphere $C_i$ created by the $i$-th blowup. Let $e_i$ be the exceptional class of the $i$-th blowup. Then we see that except for $\tilde{C}_L=C_L$, all other $\tilde{C}_i$ components must have class $e_i-\sum_{j\in \Lambda_i} e_j$ for some non-empty index set $\Lambda_i\subseteq \{i+1,\cdots,L\}$. In particular, $[\tilde{C}_i]^2\leq -2$ for all $i\leq L-1$.

Motivated by the process of resolving a $(p,q)$-cusp curve singularity, one can introduce the \emph{resolution fiber class}
\[\tilde{F}_k:=F_k-\sum_{i=1}^L m_ie_i\in H_2(X\# L\overline{\CP}^2;\mathbb{Z}).\]
The following lemma was used in \cite{LN25} to construct affine rulings of divisor complements; for the reader's convenience, we sketch its proof below.

\begin{lemma}\label{lem:resolutionfiberclass}
 Suppose $X$ is a rational manifold and the toric blowup symplectic form $\tilde{\omega}$ satisfies $[\tilde{\omega}]\cdot(K_{\tilde{\omega}}-\tilde{F}_k)<0$. Assume $\Delta_1,\cdots,\Delta_k>0,\Delta_{k+1}\leq 0$ and $C_L$ is the only $(-1)$-component in $\tilde{\st}$. Then $\tilde{F}_k$ can be represented by an embedded symplectic sphere of self-intersection $0$ intersecting $\tilde{\st}$ exactly once at the $C_L$-component transversely. In particular, $\tilde{F}_k$ is a fiber class of $(X\# L\overline{\CP}^2,\tilde{\omega})$.
\end{lemma}

\begin{proof}
 It follows from the equalities above that $\tilde{F}_k^2=0$ and $K_{\tilde{\omega}}\cdot \tilde{F}_k=-2$. By \cite[Lemma 4.1.3]{McduffSiegel}, we also know $\tilde{F}_k\cdot [C_L]=1$ and $\tilde{F}_k$ has intersection pairing $0$ with all other components in $\tilde{\st}$. By \cite[Corollary 2.8]{LN25} and the assumption $[\tilde{\omega}]\cdot(K_{\tilde{\omega}}-\tilde{F}_k)<0$, the class $\tilde{F}_k$ has non-trivial Seiberg--Witten invariant. Also, \cite[Lemma 3.2]{LN25} enables us to write $\tilde{F}_k$ as a positive linear combination of components in $\tilde{\st}$. Since the only $(-1)$-component in $\tilde{\st}$ is $C_L$, positivity of intersection implies that $\tilde{F}_k\cdot E\geq 0$ for any $E\in \mathcal{E}_{\tilde{\omega}}$. Hence, the resolution fiber class $\tilde{F}_k$ satisfies all the assumptions for being $\tilde{S}$-good (Definition \ref{def:Dgood}) and Theorem \ref{thm:MO15} implies that there exists an almost complex structure $J\in \mathcal{J}(\tilde{\st})$ such that $\tilde{F}_k$ can be represented by an embedded $J$-holomorphic sphere only intersecting the $C_L$-component once transversely.
\end{proof}

\begin{example}
  Consider a symplectic string $\st=S_1\cup \cdots\cup S_5\subseteq (X,\omega)$ with self-intersection sequence $(-3,-2,-1,-1,-2)$. Then we have $\Delta_1=1,\Delta_2=3,\Delta_3=5,\Delta_4=2,\Delta_5=-3$ and the pair $(p,q)=(-\Delta_5,\Delta_4)=(3,2)$ whose weight sequence $\mathcal{W}(p,q)=(2,1,1)$. Therefore, we can perform three toric blowups to obtain the string \[\tilde{S}_1\cup\tilde{S}_2\cup\tilde{S}_3\cup\tilde{S}_4\cup \tilde{C}_2\cup C_3\cup \tilde{C}_1 \cup \tilde{S}_5\subseteq (X\#3\overline{\CP}^2,\tilde{\omega})\] where $[\tilde{C}_1]=e_1-e_2-e_3,[\tilde{C}_2]=e_2-e_3,[C_3]=e_3$. By Lemma \ref{lem:resolutionfiberclass}, the resolution class \[F:=[S_1]+3[S_2]+5[S_3]+2[S_4]-2e_1-e_2-e_3\] will be a fiber class in $(X\#3\overline{\CP}^2,\tilde{\omega})$. See the figure below.
\end{example}

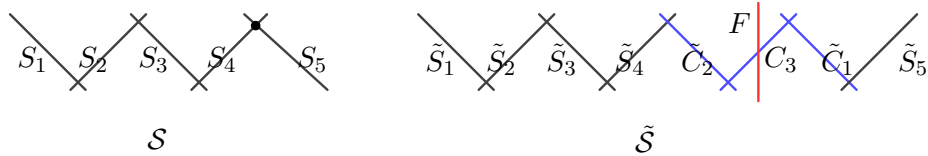
\begin{figure}[htbp]
  \centering
  \begin{tikzpicture}[x=1cm,y=1cm,line cap=round,line join=round]
    \tikzset{
      string/.style={draw=black!75, line width=0.9pt},
      blow/.style={draw=blue!70, line width=0.9pt},
      fiber/.style={draw=red!80, line width=0.9pt}
    }

    \draw[string] (0,1) -- (1,0);
    \draw[string] (0.8,0) -- (1.8,1);
    \draw[string] (1.6,1) -- (2.6,0);
    \draw[string] (2.4,0) -- (3.4,1);
    \draw[string] (3.1,1) -- (4.2,0);
    \filldraw[black] (3.25,0.85) circle (1.6pt);

    \node at (0.3,0.4) {$S_1$};
    \node at (1.1,0.4) {$S_2$};
    \node at (1.9,0.4) {$S_3$};
    \node at (2.8,0.4) {$S_4$};
    \node at (4,0.4) {$S_5$};
    \node at (1.95,-0.65) {$\st$};

    \begin{scope}[xshift=5.4cm]
     \draw[string] (0,1) -- (1,0);
    \draw[string] (0.8,0) -- (1.8,1);
    \draw[string] (1.6,1) -- (2.6,0);
    \draw[string] (2.4,0) -- (3.4,1);
      \draw[blow] (3.2,1) -- (4.2,0);
      \draw[blow] (4,0) -- (5,1);
      \draw[fiber] (4.5,1.15) -- (4.5,-0.15);
      \draw[blow] (4.8,1) -- (5.8,0);
      \draw[string] (5.6,0) -- (6.6,1);

      \node at (0.3,0.4) {$\tilde{S}_1$};
    \node at (1.1,0.4) {$\tilde{S}_2$};
    \node at (1.9,0.4) {$\tilde{S}_3$};
    \node at (2.8,0.4) {$\tilde{S}_4$};

      \node at (3.70,0.40) {$\tilde{C}_2$};
      \node at (4.25,0.88) {$F$};
      \node at (4.8,0.4) {$C_3$};
      \node at (5.55,0.4) {$\tilde{C}_1$};
      \node at (6.55,0.4) {$\tilde{S}_5$};
      \node at (3.00,-0.65) {$\tilde{\st}$};
    \end{scope}
  \end{tikzpicture}
  \caption{The toric blowups from the preceding example: the original string $\st$ on the left and the blown-up string $\tilde{\st}$ on the right, together with the fiber class $F$.}
  \label{fig:resolution-fiber-example}
\end{figure}

\subsection{Symplectic log Kodaira dimension \texorpdfstring{$-\infty$}{-infty}}\label{sec:LN25}

In this section, we review several key notions and results from \cite{LN25} concerning connected divisors $D$ in a symplectic rational manifold $(X,\omega)$ satisfying $[\omega]\cdot (K_\omega+[D])<0$. Recall first that if $(X,\omega)$ is non-minimal, then Gromov's compactness theorem implies that
\[
\inf\{\omega(E)\mid E\in\mathcal{E}_\omega\}>0,
\]
and this infimum is attained by classes in a nonempty subset $\mathcal{E}_{\mathrm{min}}\subseteq\mathcal{E}_\omega$. To analyze connected divisors with symplectic log Kodaira dimension $-\infty$, \cite{LN25} introduced the following notion, which helps reduce the complexity of the divisor configuration.

\begin{definition}\label{def:quasimin}
    A connected pair $(X,\omega,D)$ is said to be \emph{quasi-minimal} if $b_2(X)\geq 3$ and there exists some $E_{\text{min}}\in\mathcal{E}_{\text{min}}$ so that $[D]\cdot E_{\text{min}}\geq 2$.
\end{definition}

\begin{proposition}[\cite{LN25}]\label{prop:minimalreduction}
  Any pair $(X,\omega,D)$ where $D$ is connected and $[\omega]\cdot(K_{\omega}+[D])<0$ can be reduced to a pair $(\check{X},\check{\omega},\check{D})$, which is either quasi-minimal or satisfies $b_2(\check{X})\leq 2$, by a sequence of toric/non-toric/half-toric/exterior blowdowns.
\end{proposition}

The following structural result for quasi-minimal pairs was proved in \cite[Section 5]{LN25} using the curve cone theorem for almost complex $4$-manifolds developed by Zhang in \cite{Zhangcurvecone}. 

\begin{proposition}[\cite{LN25}]\label{prop:-K-D=Emin}
   Let $(X,\omega,D)$ be a connected pair with $[\omega]\cdot(K_{\omega}+[D])<0$ where $X$ is a rational manifold. If it is quasi-minimal with $[D]\cdot E_{\text{min}}\geq 2$ for some $E_{\text{min}}\in\mathcal{E}_{\text{min}}$, then $E_{\text{min}}=-[D]-K_{\omega}$ and $E_{\text{min}}\cdot [D]=2$.
\end{proposition}

A useful fact, first appearing in \cite[Lemma 1.2]{Pinsonnault}, states that when $b_2(X)\geq 3$, the class $E_{\text{min}}$ always admits an embedded $J$-holomorphic representative $C_{E_{\text{min}}}$ for \emph{any} $\omega$-tame almost complex structure $J$. Therefore, for a quasi-minimal pair $(X,\omega,D)$ with $[\omega]\cdot(K_{\omega}+[D])<0$, after choosing a divisor-adapted almost complex structure $J\in \mathcal{J}(D)$, one of the following two cases occurs.
\begin{enumerate}[nosep]
  \item $C_{E_{\text{min}}}$ is not a component of $D$: $(X,\omega,D)$ is called \emph{quasi-minimal of first kind}.
  \item $C_{E_{\text{min}}}$ is a component of $D$: $(X,\omega,D)$ is called \emph{quasi-minimal of second kind}.
\end{enumerate}

Based on this dichotomy, \cite[Section 5]{LN25} further established the following refined structural results for the two kinds.

\begin{proposition}[\cite{LN25}]\label{prop:twokinds}
  For connected pair $(X,\omega,D)$ with $[\omega]\cdot(K_\omega+[D])<0$,
  \begin{itemize}[nosep]
    \item if it is quasi-minimal of first kind, then $D$ must be a string and $D\cup C_{E_{\text{min}}}$ will be a cycle of spheres (i.e. log Calabi-Yau divisor);
    \item if it is quasi-minimal of second kind and $b_2(X)\geq 4$, then it can be reduced to a pair with $b_2=2$ by successive toric/non-toric/half-toric/exterior blowdowns.
  \end{itemize}
\end{proposition}

\begin{remark}
  In \cite{LN25}, the notion \emph{`partially minimal'} was introduced to further reduce the complexity for a quasi-minimal pair. The partially minimal condition forbids the existence of non-toric or toric exceptional spheres. It was shown that quasi-minimal pairs of second kind with $b_2(X)\geq 4$ can never be partially minimal.
\end{remark}

\section{Symplectic minimal resolutions of weighted projective planes}\label{section:WP}

Our aim in this section is to prove the `only if' part of Theorem \ref{thm:main}. We start by reviewing some standard facts about resolution of cyclic quotient singularity and refer to Evans' book \cite[Section 4.5]{Evansbook} and his blog \cite{EvansBlog} for an introduction from the symplectic toric perspective.

Let $(r,p,q)$ be positive integers with $1\leq p,q\leq r-1$. Recall that a cyclic quotient singularity $Y$ is said to be of type
\(
  \frac{1}{r}(p,q)
\)
if it is locally isomorphic to the standard model
\(
  \C^{2}/\Gamma_{r,p,q},
\)
where $\Gamma_{r,p,q}\cong\Z_{r}$ is generated by
\[
  (u,v)\longmapsto (\zeta^p u,\zeta^{q}v),\qquad \zeta=e^{2\pi i/r},\quad \gcd(r,p)=\gcd(r,q)=1.
\]
 Two parameters $q,q'\in\{1,\dots,r-1\}$ determine the same isomorphism type $\frac{1}{r}(p,q)\cong\frac{1}{r}(1,q')$ if and only if
\(
  qq'\equiv p\pmod r.
\)
Indeed, if $\Gamma_{r,p,q}=\langle g\rangle$ with
\(
  g(u,v)=(\zeta^p u,\zeta^q v),
\)
then the coordinate swap
\(
  \sigma(u,v)=(v,u)
\)
conjugates this action to
\(
  \sigma g \sigma^{-1}(u,v)=(\zeta^q u,\zeta^p v).
\)
Setting $\eta:=\zeta^q$ (still a primitive $r$-th root) and using $qq'\equiv p\pmod r$, we get
\(
  \zeta=\eta^{q'},
\)
so the conjugated action is generated by
\(
  (u,v)\mapsto(\eta u,\eta^{q'}v),
\)
which is exactly the local model of $\frac{1}{r}(1,q')$. 

Assume $p=1$. Let
\[
\frac{r}{q} = [b_1,\dots,b_k]
\]
be the negative continued fraction expansion with $b_i\ge 2$. Then the $\frac{1}{r}(1,q)$-type singularity $Y$ admits a minimal resolution $\tilde{Y}\rightarrow Y$ whose exceptional locus consists of a Hirzebruch--Jung string
\(
\st=S_1 \cup \cdots \cup S_k,
\)
with
\[
[S_i]^2 = -b_i, \qquad [S_i]\cdot [S_{i+1}]=1,
\]
and no other intersections. For the singularity type $\frac{1}{r}(1,q')$ with $qq'\equiv 1\pmod r$, the minimal resolution is the same Hirzebruch--Jung string with the orientation reversed by the standard duality phenomenon of continued fractions.

We equip $\C^2$ with the standard symplectic form and the Poisson-commuting Hamiltonians
\[
H_1(u,v)=\tfrac12|u|^2,
\qquad
H_2(u,v)=\tfrac12|v|^2.
\]
Since both are $\Gamma_{r,1,q}$-invariant, they descend to the quotient orbifold $\C^2/\Gamma_{r,1,q}$ (still denoted by $H_1,H_2$). To obtain an effective Hamiltonian $T^2$-action on $\C^2/\Gamma_{r,1,q}$, we choose the moment map
\[
\mu:\C^2/\Gamma_{r,1,q}\to \R^2,
\qquad
\mu(u,v)=\bigl(H_2,\tfrac1r(H_1+qH_2)\bigr).
\]
Its image is the non-Delzant corner
\[
\mathcal{O}_{r,q}:=\{(x,y)\in\R^2\mid x\ge 0,\ y\ge \tfrac{q}{r}x\}.
\]
In this toric picture, the HJ string $\st$ arises from the corner-chopping procedure below, which resolves the non-Delzant corner $\mathcal{O}_{r,q}$ into a Delzant one.
\begin{itemize}[nosep]
  \item First truncate the corner $\mathcal{O}_{r,q}$ by a horizontal line $y=\varepsilon$, producing two new vertices
  \(
    V_1=(0,\varepsilon), V_1'=(\frac{r}{q}\varepsilon,\varepsilon).
  \)
  The vertex $V_1$ is Delzant. If $V_1'$ is also Delzant, then the process stops.
  \item Otherwise, apply an $\text{AGL}(2;\Z)$-transformation sending a neighborhood of $V_1'$ to the corner $\mathcal{O}_{q,b_1q-r}$. Truncate again by a horizontal line to obtain vertices $V_2$ and $V_2'$, where $V_2$ is Delzant. If $V_2'$ is Delzant, stop; otherwise, repeat the same procedure at $V_2'$.
  \item After finitely many steps, we obtain vertices $V_k$ and $V_k'$ that are both Delzant. The original non-Delzant corner is then replaced by an open Delzant polygon with $k$ compact edges, corresponding to the spheres $S_1,\cdots,S_k$.
\end{itemize}

 Now, let $(a,b,c)$ be pairwise coprime positive integers such that $a<b<c$. Consider the symplectic orbifold $\CP(a,b,c)$ which has three isolated cyclic quotient singularities of types
\[
\frac{1}{a}(b,c)\cong\frac{1}{a}(1,a_b)\cong\frac{1}{a}(1,a_c), \quad \frac{1}{b}(a,c)\cong\frac{1}{b}(1,b_a)\cong\frac{1}{b}(1,b_c), \quad \frac{1}{c}(a,b)\cong\frac{1}{c}(1,c_b)\cong\frac{1}{c}(1,c_a),
\]
where $0<a_b,a_c<a,0<b_a,b_c<b,0<c_a,c_b<c$ are chosen such that
\[
a_b b\equiv c,
\,
a_c c\equiv b\pmod a,
\quad
b_a a\equiv c,
\,
b_c c\equiv a\pmod b,
\quad
c_a a\equiv b,
\,
c_b b\equiv a\pmod c.
\]
Recall that, as a symplectic toric orbifold, $\CP(a,b,c)$ is constructed by symplectic reduction as the $S^1$-quotient of the boundary of some ellipsoid
  \[\{(z_0,z_1,z_2)\,|\,\frac{1}{2}(a|z_0|^2+b|z_1|^2+c|z_2|^2)=\tau\}/S^1\] for some constant $\tau$.
  The residual $T^{2}$-action endows $\WCP(a,b,c)$ with a toric structure whose moment polygon is a labeled triangle embedded in $\{(x,y,z)\in \R^{3}\,|\,x,y,z\geq 0\}$ satisfying
\[
 ax+by+cz=\tau.
\]
After rescaling, we set $\tau=abc$ and choose certain convenient planar presentations of this triangle with vertices
\[
  A=(bc,0,0),\qquad B=(0,ac,0),\qquad C=(0,0,ab)
\]
 by some integral affine transformation. The integral basis $(\vec{m}_i,\vec{n}_i)$ for the rank-$2$ lattice orthogonal to $(a,b,c)^\text{T}$ can be chosen as follows for each $1\leq i \leq 6$:
 \[\vec{m}_1=\begin{pmatrix}\frac{a_bb-c}{a}\\-a_b\\1\end{pmatrix}
  \vec{m}_2=\begin{pmatrix}\frac{a_cc-b}{a}\\1\\-a_c\end{pmatrix}
  \vec{m}_3=\begin{pmatrix}-b_a\\ \frac{b_aa-c}{b}\\1\end{pmatrix}
   \vec{m}_4=\begin{pmatrix}1\\ \frac{b_cc-a}{b}\\-b_c\end{pmatrix}
    \vec{m}_5=\begin{pmatrix}-c_a\\1\\ \frac{c_aa-b}{c}\end{pmatrix}
    \vec{m}_6=\begin{pmatrix}1\\ -c_b\\ \frac{c_bb-a}{c}\end{pmatrix}\]
 \[\vec{n}_1=\begin{pmatrix}-b\\a\\0\end{pmatrix}
 \vec{n}_2=\begin{pmatrix}-c\\0\\a\end{pmatrix}
 \vec{n}_3=\begin{pmatrix}b\\-a\\0\end{pmatrix}
 \vec{n}_4=\begin{pmatrix}0\\-c\\b\end{pmatrix}
 \vec{n}_5=\begin{pmatrix}c\\0\\-a\end{pmatrix}
 \vec{n}_6=\begin{pmatrix}0\\c\\-b\end{pmatrix}.\] 
Under the basis $(\vec{m}_1,\vec{n}_1)$, $\overrightarrow{AB}=c\vec{n}_1,\overrightarrow{AC}=ab\vec{m}_1+a_bb\vec{n}_1$. Thus, the moment map image can be presented as the triangle $\mathcal{T}_1\subseteq \R^2$ with vertices
\[A_1=(0,0),\quad B_1=(0,c),\quad C_1=(ab,a_bb)\]
where $A_1,B_1,C_1$ correspond to $A,B,C$ respectively. Similarly, by using other basis $(\vec{m}_i,\vec{n}_i)$ for $2\leq i\leq 6$, we may obtain different presentation $\mathcal{T}_i\subseteq \R^2$ with vertices
\begin{align*}
  A_2=(0,0),\quad C_2=(0,b),\quad B_2=(ac,a_cc)\\
  B_3=(0,0),\quad A_3=(0,c),\quad C_3=(ba,b_aa)\\
  B_4=(0,0),\quad C_4=(0,a),\quad A_4=(bc,b_cc)\\
  C_5=(0,0),\quad A_5=(0,b),\quad B_5=(ca,c_aa)\\
  C_6=(0,0),\quad B_6=(0,a),\quad A_6=(cb,c_bb)
\end{align*}
where $A_i,B_i,C_i$ correspond to $A,B,C$ respectively.
\begin{example}
Let us take $(a,b,c)=(2,3,5)$ and consider the toric structure on $\CP(2,3,5)$. Then
\[
a_b=a_c=1,\qquad b_a=b_c=1,\qquad c_a=c_b=4.
\]
The corresponding moment polygons in $\R^2$ are illustrated in Figure~\ref{fig:moment-polygons-235}.
\begin{figure}[htbp]
  \centering
  \resizebox{0.9\textwidth}{!}{%
  \begin{tikzpicture}[x=0.52cm,y=0.52cm,line cap=round,line join=round,font=\small]
    \tikzset{
      tri/.style={draw=black!65,fill=black!12,line width=0.9pt},
      pt/.style={circle,fill=black,inner sep=1.5pt}
    }

    \coordinate (A1) at (0,0);
    \coordinate (B1) at (0,5);
    \coordinate (C1) at (6,3);
    \filldraw[tri] (A1)--(B1)--(C1)--cycle;
    \node[pt,label={[below left]$A_1(0,0)$}] at (A1) {};
    \node[pt,label={[above left]$B_1(0,5)$}] at (B1) {};
    \node[pt,label={[right]$C_1(6,3)$}] at (C1) {};
    \node at (1.8,2.5) {$\mathcal{T}_1$};

    \coordinate (A2) at (8,7);
    \coordinate (C2) at (8,10);
    \coordinate (B2) at (16,12);
    \filldraw[tri] (A2)--(C2)--(B2)--cycle;
    \node[pt,label={[below left]$A_2(0,0)$}] at (A2) {};
    \node[pt,label={[above left]$C_2(0,3)$}] at (C2) {};
    \node[pt,label={[right]$B_2(10,5)$}] at (B2) {};
    \node at (10.2,9.4) {$\mathcal{T}_2$};

    \coordinate (B3) at (17,5);
    \coordinate (A3) at (17,10);
    \coordinate (C3) at (22.5,7);
    \filldraw[tri] (B3)--(A3)--(C3)--cycle;
    \node[pt,label={[below]$B_3(0,0)$}] at (B3) {};
    \node[pt,label={[above]$A_3(0,5)$}] at (A3) {};
    \node[pt,label={[right]$C_3(6,2)$}] at (C3) {};
    \node at (19.2,7.3) {$\mathcal{T}_3$};

    \coordinate (O) at (12,2);
    \coordinate (A) at (9.0,1.0);
    \coordinate (B) at (15,1.0);
    \coordinate (C) at (12,5);
    \filldraw[tri] (A)--(B)--(C)--cycle;
    \draw[dashed,gray!70,line width=0.8pt,->] (O)--(8.1,0.7);
    \draw[dashed,gray!70,line width=0.8pt,->] (O)--(16.2,0.6);
    \draw[dashed,gray!70,line width=0.8pt,->] (O)--(12,7);
    \node[pt,label={[below]$A(15,0,0)$}] at (A) {};
    \node[pt,label={[below]$B(0,10,0)$}] at (B) {};
    \node[pt,label={[above]$C(6,0,0)$}] at (C) {};

    \coordinate (B4) at (-4,-8);
    \coordinate (C4) at (-4,-6);
    \coordinate (A4) at (10,-3);
    \filldraw[tri] (B4)--(C4)--(A4)--cycle;
    \node[pt,label={[below left]$B_4(0,0)$}] at (B4) {};
    \node[pt,label={[above left]$C_4(0,2)$}] at (C4) {};
    \node[pt,label={[above]$A_4(15,5)$}] at (A4) {};
    \node at (-0.7,-6) {$\mathcal{T}_4$};

    \coordinate (C5) at (9,-9);
    \coordinate (A5) at (9,-6);
    \coordinate (B5) at (17,-1);
    \filldraw[tri] (C5)--(A5)--(B5)--cycle;
    \node[pt,label={[below left]$C_5(0,0)$}] at (C5) {};
    \node[pt,label={[left]$A_5(0,3)$}] at (A5) {};
    \node[pt,label={[above]$B_5(10,8)$}] at (B5) {};
    \node at (11.2,-5.8) {$\mathcal{T}_5$};

    \coordinate (C6) at (13.4,-9);
    \coordinate (B6) at (13.4,-7);
    \coordinate (A6) at (27,-1);
    \filldraw[tri] (C6)--(B6)--(A6)--cycle;
    \node at (12,-7.5) {$B_6(0,2)$};
    \node[pt,label={[below left]$C_6(0,0)$}] at (C6) {};
    \node[pt,label={}] at (B6) {};
    \node[pt,label={[above]$A_6(15,12)$}] at (A6) {};
    \node at (18.3,-5.5) {$\mathcal{T}_6$};
  \end{tikzpicture}%
  }
  \caption{Six planar presentations $\mathcal{T}_1,\dots,\mathcal{T}_6$ of the moment triangle for $\CP(2,3,5)$.}
  \label{fig:moment-polygons-235}
\end{figure}
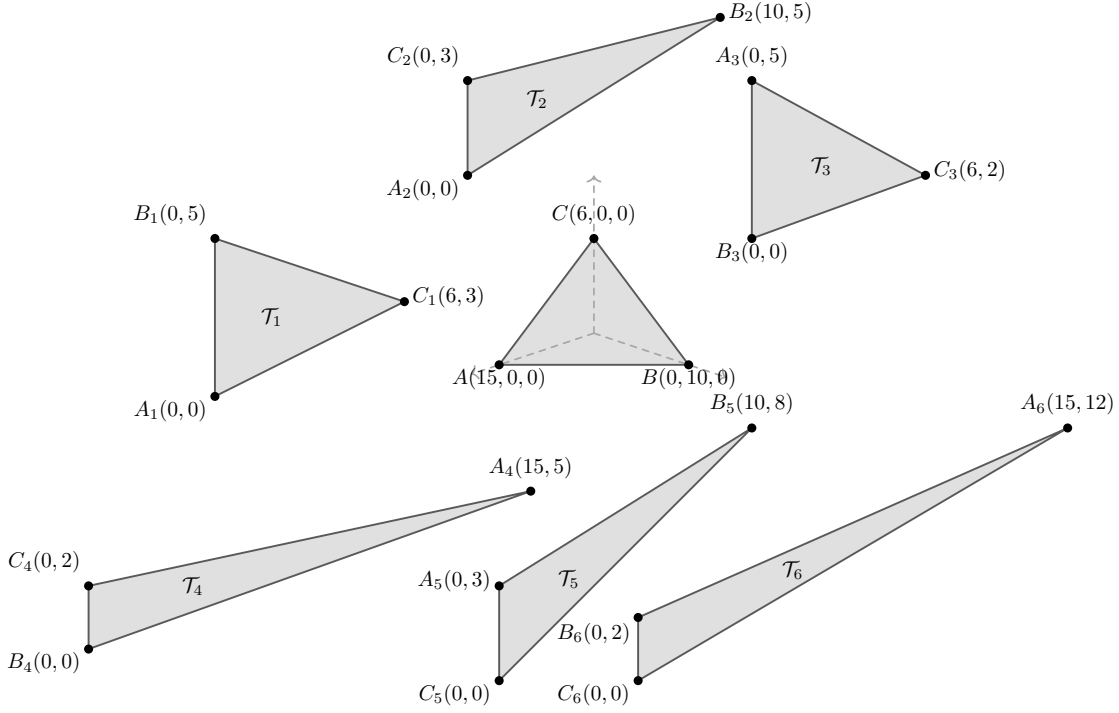
\end{example}

Thus, starting from any presentation $\mathcal{T}_i$ of the moment polygon of $\CP(a,b,c)$, we can truncate the three non-Delzant corners using the previous procedure. In this way, the non-Delzant triangle $\mathcal{T}_i$ is converted into a Delzant polygon $\mathcal{P}$, which is the moment polygon of a smooth symplectic toric $4$-manifold $(X,\omega)$. The symplectic form $\omega$ depends on the truncation sizes, whereas the diffeomorphism type of $X$ is uniquely determined by the triple $(a,b,c)$. Indeed, resolving the three singularities produces three Hirzebruch--Jung strings $\st_a,\st_b,\st_c$, determined by
\[
\frac{a}{a_b}=[a_1,\dots,a_{k_a}],\quad
\frac{b}{b_c}=[b_1,\dots,b_{k_b}],\quad
\frac{c}{c_a}=[c_1,\dots,c_{k_c}],
\]
where $k_a,k_b,k_c$ are exactly the numbers of new edges created near the three corners. In addition, $\mathcal{P}$ has three edges inherited from the original triangle; their moment map preimages are spheres $N_a,N_b,N_c$, each disjoint from $\st_a,\st_b,\st_c$, respectively. Setting
\(
n:=k_a+k_b+k_c,
\)
we see that $\mathcal{P}$ has $n+3$ edges, and hence
\(
X\cong \CP^2\#n\overline{\CP}^2.
\)
We call any such pair $(X,\omega,D:=\st_a\cup\st_b\cup\st_c)$ a \emph{symplectic minimal resolution} of $\CP(a,b,c)$. In the terminology of Definition \ref{def:3assumptions}, $D$ may also be called a sub-toric configuration.

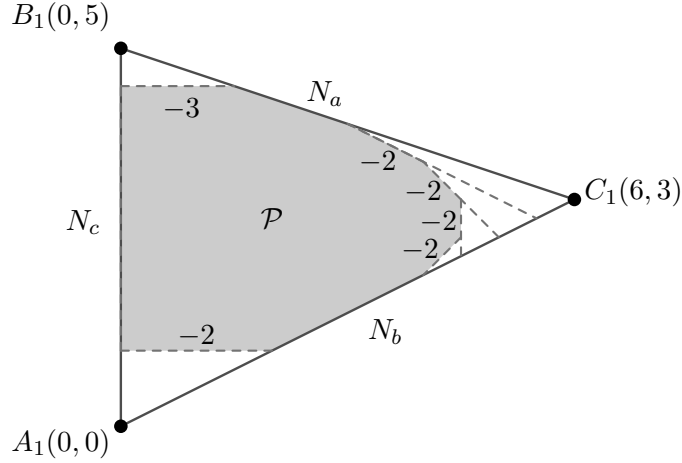
\begin{figure}[htbp]
  \centering
  \begin{tikzpicture}[x=1cm,y=1cm, line cap=round, line join=round]
    \tikzset{
      boundary/.style={draw=black!70, line width=0.9pt},
      trunc/.style={draw=black!55, dashed, line width=0.8pt},
      fillpoly/.style={fill=black!20, draw=none},
      pt/.style={circle, fill=black, inner sep=1.8pt}
    }

    \coordinate (A) at (0,0);
    \coordinate (B) at (0,5);
    \coordinate (C) at (6,3);

    \coordinate (L1) at (0,1.0);
    \coordinate (L2) at (0,4.5);
    \coordinate (M1) at (2.,1.0);
    \coordinate (M2) at (1.5,4.5);
    \coordinate (R0) at (4,2);
    \coordinate (R1) at (4.5,2.5);
    \coordinate (R2) at (4.5,3);
    \coordinate (R3) at (4,3.5);
    \coordinate (R4) at (3,4);
    \coordinate (S1) at (5.5,2.75);
    \coordinate (S2) at (5,2.5);
    \coordinate (S3) at (4.5,2.25);

    \fill[fillpoly] (L1) -- (M1) -- (R0) -- (R1) -- (R2) -- (R3) -- (R4) -- (M2) -- (L2) -- cycle;

    \draw[boundary] (A) -- (B) -- (C) -- cycle;
    \draw[trunc] (L1) -- (L2);
    \draw[trunc] (L1) -- (M1);
    \draw[trunc] (L2) -- (M2);
    \draw[trunc] (R4) -- (S1);
    \draw[trunc] (R1) -- (R0);
    \draw[trunc] (S2) -- (R2);
    \draw[trunc] (R2) -- (R3);
    \draw[trunc] (R2) -- (S3);
    \draw[trunc] (R3) -- (R4);

    \node[pt,label={[below left]$A_1(0,0)$}] at (A) {};
    \node[pt,label={[above left]$B_1(0,5)$}] at (B) {};
    \node[pt,label={[right]$C_1(6,3)$}] at (C) {};

    \node at (-0.5,2.7) {$N_c$};
    \node at (2.7,4.4) {$N_a$};
    \node at (3.5,1.25) {$N_b$};
    \node at (2.0,2.75) {$\mathcal{P}$};

    \node at (0.8,4.2) {$-3$};
    \node at (1.0,1.2) {$-2$};
    \node at (3.4,3.5) {$-2$};
    \node at (4,3.1) {$-2$};
    \node at (4.2,2.7) {$-2$};
    \node at (3.96,2.34) {$-2$};
  \end{tikzpicture}
  \caption{Moment polygon of the minimal resolution $\CP^2\#6\overline{\CP}^2$ from the planar presentation $\mathcal{T}_1$ of $\CP(2,3,5)$, with the three inherited edges labeled by $N_a$, $N_b$ and $N_c$, together with the truncations producing three Hirzebruch--Jung strings with self-intersection sequences $(-2),(-3),(-2,-2,-2,-2)$.}
  \label{fig:lemma-2-8-diagram}
\end{figure}

\begin{remark}\label{rmk:chen}
 In \cite[Theorem 2.6]{Chen09} (see also \cite{Mcduff09}), Chen proved that any two cohomologous symplectic forms on the weighted projective plane $\CP(a,b,c)$ are diffeomorphic. The argument uses equivariant pseudoholomorphic curve techniques and extends the Gromov--Taubes' uniqueness theorem for symplectic structures on $\CP^2$. Consequently, when discussing the symplectic minimal resolution of $\CP(a,b,c)$, we may identify it with the resolution of a symplectic reduction model, without having to specify a particular symplectic form on the orbifold.
\end{remark}

\begin{lemma}\label{lem:threeedges}
  If $a<b<c$, then $[N_a]^2=[N_b]^2=-1$ and $[N_c]^2\geq -1$.
\end{lemma}
\begin{proof}
  To compute $[N_a]^2$, it is convenient to choose either $\mathcal{T}_4$ or $\mathcal{T}_6$ as the model for truncation. Let us take $\mathcal{T}_4$ as an example. First, by truncating the $B_4$-corner which can be locally identify with the $\mathcal{O}_{bc,b_cc}$-model, we obtain a horizontal edge adjacent to the vertical edge of the initial triangle. To deal with the $C_4$-corner, apply the $\text{GL}(2;\mathbb{Z})$-transformation by 
  \[G:=\begin{pmatrix}
    1 & 0 \\
    1 & -1
  \end{pmatrix}\]
  so that $A_4,B_4,C_4$ are respectively mapped to 
  \[A_4'=(bc,bc-b_cc),\qquad B_4'=(0,0),\qquad C_4'=(0,-a).\]
  By our assumption on $a<b<c$ and the choice of $b_c$, it is straightforward to see that \[0<bc-b_cc+a<bc.\]
  This enables us to locally identify the $C_4'$-corner with the $\mathcal{O}_{bc,bc-b_cc+a}$-model so that we can apply the principle of corner chopping described above to obtain a horizontal edge adjacent to the vertical edge. Returning to the original triangle $\Delta_{A_4B_4C_4}$, this vertical edge is mapped by the inverse $\text{GL}(2,\mathbb{Z})$-transformation to an edge in the direction $(1,1)$. Standard fact in toric geometry tells us that the sphere $N_a$ lying over the vertical edge has self-intersection $-1$.

  The computations of $[N_b]^2$ and $[N_c]^2$ are analogous. For $N_b$, we may work with either $\mathcal{T}_2$ or $\mathcal{T}_5$; take $\mathcal{T}_2$ and truncate its $A_2$-corner so that a horizontal edge is attached to the vertical edge. Since
  \[0<ac-a_cc+b<ac,\]
  we can again use the matrix $G$ to identify the $C_2$-corner with the $\mathcal{O}_{ac,ac-a_cc+b}$-model. After corner chopping, the edge adjacent to the vertical edge of $\Delta_{A_2B_2C_2}$ has direction $(1,1)$. Therefore $[N_b]^2=-1$. For $N_c$, we may work with either $\mathcal{T}_1$ or $\mathcal{T}_3$; take $\mathcal{T}_1$ and chop its $A_1$-corner to obtain one horizontal edge attached to the vertical edge. Consider the unique integer $t$ such that 
  \[0<tab-a_bb+c<ab.\] It is clear that $t\leq 1$. Then we apply the $\text{GL}(2;\mathbb{Z})$-transformation by the matrix
  \[G_t:=\begin{pmatrix}
    1 & 0 \\
    t & -1
  \end{pmatrix}\]
  to identify the $B_1$-corner with the $\mathcal{O}_{ab,tab-a_bb+c}$-model. Since $G_t^{-1}$ sends the horizontal direction $(1,0)^{\text{T}}$ to $(1,t)^{\text{T}}$ with $t\leq 1$, we see that $[N_c]^2=-t\geq -1$.
\end{proof}

\begin{corollary}\label{cor:3nminus6}
  \[\sum_{\alpha=1}^{k_a} a_\alpha+\sum_{\beta=1}^{k_b} b_{\beta}+\sum_{\gamma=1}^{k_c} c_{\gamma}\geq 3n-6.\]
\end{corollary}
\begin{proof}
  Since $-K_{\omega}$ is represented by $D\cup N_a\cup N_b\cup N_c$, by Lemma \ref{lem:threeedges} we have
   \begin{align*}
  9-n&=K_\omega^2=\big([D]+[N_a]+[N_b]+[N_c]\big)^2\\
  &=-\big(\sum_{\alpha=1}^{k_a} a_\alpha+\sum_{\beta=1}^{k_b} b_{\beta}+\sum_{\gamma=1}^{k_c} c_{\gamma}\big)+[N_a]^2+[N_b]^2+[N_c]^2+2(n+3)\\
  &\geq -\big(\sum_{\alpha=1}^{k_a} a_\alpha+\sum_{\beta=1}^{k_b} b_{\beta}+\sum_{\gamma=1}^{k_c} c_{\gamma}\big)-3+2(n+3).
  \end{align*}
\end{proof}

Now, we can prove one direction in our main Theorem \ref{thm:main}.
\begin{proposition}\label{prop:onlyif}
  Let $(X,\omega,D=\st_a\cup\st_b\cup\st_c)$ be a pair coming from the symplectic minimal resolution of $\CP(a,b,c)$ for pairwise coprime $(a,b,c)$ with $a<b<c$. Then 
\[\mathcal{E}^{\text{log}}_{\omega}(\st_b,\st_c;D)=\{[N_a]\},\quad \mathcal{E}^{\text{log}}_{\omega}(\st_a,\st_c;D)=\{[N_b]\},\quad \mathcal{E}^{\text{log}}_{\omega}(\st_a,\st_b;D)=\begin{cases}\{[N_c]\},\text{if } [N_c]^2=-1,\\
\varnothing ,\text{if } [N_c]^2\geq 0.\end{cases}\]
In particular, $D$ is full, of $(a,b,c)$-type and gap-admissible.
\end{proposition}

\begin{proof}
On the one hand, Lemma \ref{lem:threeedges} immediately gives
\[
  [N_a]\in \mathcal{E}^{\text{log}}_{\omega}(\st_b,\st_c;D),\quad [N_b]\in \mathcal{E}^{\text{log}}_{\omega}(\st_a,\st_c;D),\quad [N_c]\in\mathcal{E}^{\text{log}}_{\omega}(\st_a,\st_b;D)\text{ when } [N_c]^2=-1.
\]
On the other hand, for any
$E\in\mathcal{E}_\omega\setminus \{[N_a],[N_b],[N_c]\}$,
the adjunction formula implies
\[-1=-2-E^2=E\cdot K_\omega=-E\cdot ([\st_a]+[\st_b]+[\st_c]+[N_a]+[N_b]+[N_c]).\]  
By \cite[Lemma 2.1]{Pinsonnault}, we have $E\cdot [N_a]\geq 0$, $E\cdot[N_b]\geq 0$, and if $[N_c]^2=-1$, $E\cdot [N_c]\geq 0$.
When $[N_c]^2\geq 0$, we still have $E\cdot [N_c]\geq 0$ by \cite[Lemma 2.2]{Zhangcurvecone}, since $E$ has nontrivial SW invariant.
By definition, if $E\in \mathcal{E}^{\text{log}}_\omega(\st_i,\st_j;D)$ for some $i\neq j\in\{a,b,c\}$, then
\[
  E\cdot([\st_a]+[\st_b]+[\st_c])\geq 1+1+0= 2.
\]
This contradicts the adjunction equality above. Therefore, $[N_a],[N_b],[N_c]$ are the only possible classes contributing to the exceptional gaps.

$D$ is clearly full and of $(a,b,c)$-type by assumptions. To see it is also gap-admissible, note that 
\[[\omega]\cdot(K_\omega+[D])=-[\omega]\cdot ([N_a]+[N_b]+[N_c])\leq -\e(\st_b,\st_c;D)-\e(\st_a,\st_c;D)-\e(\st_a,\st_b;D).\]
\end{proof}

\begin{proof}[Proof of `only if' part of Theorem \ref{thm:main}]
  By Proposition \ref{prop:onlyif}, it is enough to prove that, under the assumptions that $D$ is sub-toric, full, and of $(a,b,c)$-type, the divisor $D$ is obtained from the symplectic minimal resolution. Suppose $D=\st_1\cup\st_2\cup\st_3$ is sub-toric. Since the toric boundary $\mu^{-1}(\mathcal{P})$ is a cycle of spheres, the complement $\mu^{-1}(\mathcal{P})\setminus D$ consists of exactly three strings, which we denote by $N_1,N_2,N_3$. Because $D$ is full, each $N_i$ must in fact be a single sphere. We may label these components so that
  \(
    [N_i]\cdot [\st_i]=0, [N_i]\cdot [\st_{i\pm 1}]=1 \text{ for }i\in\{1,2,3\}
  \), where $i\pm 1$ is understood mod $3$.
  The $(a,b,c)$-type assumption therefore implies that, for each $i$, a neighborhood of $\mu(\st_i)$ in $\mathcal{P}$ is identified with the truncation of the corresponding orbifold corner of the moment triangle of $\CP(a,b,c)$. Indeed, in a Delzant polygon, a neighborhood of a chain of edges is determined by the associated self-intersection sequence, up to rescaling the edge lengths. Hence, $D$ is precisely the divisor arising from the symplectic minimal resolution.
\end{proof}

If a string only contains $(-2)$-spheres, we will call it a \emph{$(-2)$-string}. The following lemma concering the case of two $(-2)$-strings will be needed later.

\begin{lemma}\label{lem:-2chains}
Let $a,b,c>1$ be pairwise coprime integers where $a>b$. Let $\st_a,\st_b,\st_c$ be three HJ strings coming from the minimal resolution of $\mathbb{CP}(a,b,c)$ at $a,b,c$-point respectively. Then $\st_a,\st_b$ are $(-2)$-strings if and only if $c=kab-a-b$ for some integer $k\ge 1$.
Moreover, in this situation, the third string $\st_c$ has self-intersection sequence
\begin{itemize}[nosep]
    \item $(-a,-k,-b)$ if $k\ge 2$; 
    \item $(1-a,1-b)$ if $k=1$ and $b\neq 2$;
    \item $(2-a)$ if $k=1$ and $b=2$.
\end{itemize}

\end{lemma}

\begin{proof}
The cyclic quotient singularity at the $a$-point has type
\[
\frac{1}{a}(b,c)\cong \frac{1}{a}(1,q), \qquad q\equiv b^{-1}c \pmod a.
\]
Hence, $\st_a$ is a $(-2)$-string if and only if $q\equiv -1 \pmod a$, i.e., $c\equiv -b \pmod a$.
Similarly, $\st_b$ is a $(-2)$-string if and only if
$c\equiv -a \pmod b$.
Since $\gcd(a,b)=1$, Chinese remainder theorem implies that these two congruences are equivalent to
\(
c\equiv -(a+b)\pmod{ab},
\)
and
\(
c=kab-a-b
\)
for some integer $k\ge 1$. This proves the equivalence.

Assume now $c=kab-a-b$. Then \(
a(kb-1)\equiv b \pmod c,
\)
since
\[
a(kb-1)=akb-a=b+(kab-a-b)=b+c.
\]
So, the cyclic quotient singularity at the $c$-point has type
\[
\frac{1}{c}(a,b)\cong \frac{1}{c}(1,kb-1).
\]
A direct computation gives
\[
[a,k,b]
=
a-\frac{1}{k-\frac{1}{b}}
=
\frac{kab-a-b}{kb-1}
=
\frac{c}{kb-1},
\]
so the third string $\st_c$ has self-intersection sequence $(-a,-k,-b)$ when $k\geq 2$. If $k=1$ and $b\neq 2$, in order for the items in the continued fraction to be at least $2$, one should instead write
\[[a-1,b-1]=(a-1)-\frac{1}{b-1}=\frac{ab-a-b}{b-1}=\frac{c}{b-1}.\]
Hence, $\st_c$ has self-intersection sequence $(1-a,1-b)$ in this case. Similarly, when $k=1$ and $b=2$, the self-intersection sequence should be further reduced to $(2-a)$.
\end{proof}

\begin{remark}
  In particular, the above lemma implies that there cannot exist three full $(-2)$-strings of $(a,b,c)$-type for pairwise coprime positive integers $(a,b,c)$. Weimin Chen (\cite{Chentalk}) has forthcoming work that also studies the maximal possible number of \emph{disjoint} symplectic $(-2)$-spheres in rational manifolds.
\end{remark}

\section{Symplectic Hirzebruch--Jung strings}\label{Section:HJstrings}

\subsection{Exceptional spheres and HJ strings}\label{section:HJnontoric}
In this section, we let $D\subseteq (X,\omega)$ be a symplectic divisor in a symplectic rational manifold, given as the disjoint union of three HJ strings $\st_1,\st_2,\st_3$ that satisfy the \emph{full} condition
\[
  l(\st_1)+l(\st_2)+l(\st_3)=b_2^-(X).
\]
Our goal is to analyze the intersection pattern between exceptional spheres and $D$. We begin with the following basic observation for exterior exceptional spheres.

\begin{lemma}\label{lem:noexterior}
For $D=\st_1\cup\st_2\cup\st_3\subseteq (X,\omega)$ which is full, there is no exterior exceptional sphere.
\end{lemma}

\begin{proof}
  Otherwise, there will be a negative definite divisor configuration of $b_2^-(X)$ components in $X$, which is a contradiction.
\end{proof}

Next, we also exclude the existence of non-toric exceptional spheres by analyzing degenerations of the fiber class in Section \ref{section:degfiberclass}.

\begin{proposition}\label{prop:nonontoric}
 For $D=\st_1\cup\st_2\cup\st_3\subseteq (X,\omega)$ which is full, there is no non-toric exceptional sphere if assuming $[\omega]\cdot K_{\omega}<0$.
\end{proposition}
  
\begin{proof}
  Suppose there was a non-toric exceptional sphere $C_E$ intersecting at the $k$-th component $A_k$ of the HJ string $\st_1=\cup_{i=1}^{n} A_i$ with negative self-intersection sequence $(a_1,a_2,\cdots,a_{n})$. Then the configuration $\st_1\cup C_E$ can not be negative definite since otherwise $D\cup C_E$ will be a negative definite configuration with $b_2^-(X)+1$ components. Let us perform non-toric blowdown of $C_E$ so that $\st_1$ descends to a non-negative definite string $\st'=\cup_{i=1}^nA_i'$ in the blowdown manifold $(X',\omega')$ with negative self-intersection sequence $(a_1,a_2,\cdots,a_{k-1},a_k-1,a_{k+1},\cdots,a_n)$, while the strings $\st_2,\st_3$ are unaffected. Thus, there exists an index $l\ge k$ such that
\[
  \Delta_i=\det M_{i-1}(\st') > 0 \quad \text{for all } i \le l,
  \qquad
  \Delta_{l+1}=\det M_{l}(\st') \le 0.
\]
By the discussion in Section \ref{section:resfiberclass}, we can  define classes \[F_l:=\sum_{i=1}^l\Delta_i[A_i']\in H_2(X';\mathbb{Z}),\qquad \tilde{F}_l:=F_l-\sum_{i=1}^Lm_ie_i\in H_2(X'\# L\overline{\CP}^2;\mathbb{Z}),\]
where $(m_1,\cdots,m_L)$ is the weight sequence for the coprime pair $(-\Delta_{l+1},\Delta_{l})$. Notice that if we choose the toric blowup symplectic form $\tilde{\omega}$ on $X'\# L\overline{\CP}^2$ to have sufficiently small areas on exceptional classes $e_1,\cdots,e_L$, then the inequality
\[[\tilde{\omega}]\cdot(K_{\tilde{\omega}}-\tilde{F_l})<0\]
will follow from the inequality
\[[\omega]\cdot(K_{\omega}-F_l)<0\] by our assumption $[\omega]\cdot K_{\omega}<0$ and the obvious inequality $[\omega]\cdot F_l>0$.
Hence, by Lemma \ref{lem:resolutionfiberclass}, $\tilde{F}_l$ is a fiber class. Denote by $\tilde{\st'}$ the total transform of $\st'$ under $L$ successive toric blowups and write
\[\tilde{\st'}=\tilde{A}_1'\cup\cdots\cup\tilde{A}_k'\cup\cdots\cup\tilde{A}'_{l}\cup \mathcal{L}\cup\tilde{A}'_{l+1}\cup \cdots\cup\tilde{A}_n'\]
where $\tilde{A}_i'$ is the proper transform of $A_i'$ and $\mathcal{L}$ is a string of length $L$ consisting of the proper transforms of exceptional spheres in the process of toric blowups. Note that $\mathcal{L}$ will contain a unique $(-1)$-component $C_L$ coming from the last blowup and all other components have self-intersection $\leq -2$. Lemma also \ref{lem:resolutionfiberclass} guarantees that for generic $J\in\mathcal{J}(\tilde{\st'})$, $\tilde{F}_l$ can be represented by embedded $J$-holomorphic spheres intersecting $\tilde{\st'}$ exactly once at the $C_L$-component.

Observe that $\tilde{\st'}\setminus C_L$ is a disjoint union of two strings $\mathcal{H}_0,\mathcal{H}_1$, and the one (say $\mathcal{H}_1$) which does not contain $\tilde{A}_k'$ is Hirzebruch--Jung. We consider the $J$-holomorphic subvarieties $\Theta_0,\Theta_1,\Theta_2,\Theta_3$ in class $\tilde{F}_l$ passing through the point on $\mathcal{H}_0,\mathcal{H}_1,\st_2,\st_3$ respectively. Since $\tilde{F}_l$ has an embedded representative which is disjoint from $\mathcal{H}_0,\mathcal{H}_1,\st_2,\st_3$, all the components in these four strings must be entirely contained as components in $\Theta_0,\Theta_1,\Theta_2,\Theta_3$ respectively, due to positivity of intersection. When $\Theta_2\neq \Theta_3$, we choose some component $B_j$ of $\Theta_j$ for each $j\in\{2,3\}$ by the following.
\begin{itemize}
  \item If $\Theta_j=\Theta_i$ for some $i\in\{0,1\}$, then there will be components in $\Theta_i=\Theta_j$ connecting $\mathcal{H}_i$ to $\st_j$, and we may take any one of them and call it $B_j$.
  \item Otherwise, since $[\tilde{F}_l]\cdot [C_L]=1$, $\Theta_j$ must contain a $(-1)$-component disjoint from $C_L$ by Proposition \ref{prop:fiberclass}(3), and we call this $(-1)$-component $B_j$. 
\end{itemize}
When $\Theta_2=\Theta_3$, we choose $B_2,B_3$ by the following.
\begin{itemize}
  \item If $\Theta_2=\Theta_3=\Theta_i$ for some $i\in\{0,1\}$, we choose $B_2,B_3$ to be some components in $\Theta_2=\Theta_3$ connecting $\st_2$ to $\st_3$ and  $\st_2\cup \st_3$ to $\mathcal{H}_i$ respectively.
  \item Otherwise, we still choose $B_2$ to be some component in $\Theta_2=\Theta_3$ connecting $\st_2$ to $\st_3$ but $B_3$ to be a $(-1)$-component in $\Theta_1$. Note that this is possible by Proposition \ref{prop:fiberclass}(3) since $\mathcal{H}_1$ is Hirzebruch--Jung.
\end{itemize}

  Now, denote by $C_{\alpha_0},C_{\alpha_1}$ the components in $\mathcal{L}$ adjacent to $C_L$ and contained in $\Theta_0,\Theta_1$ respectively. To complete the proof, we analyze the following configuration (see Figure \ref{fig:nonontoric} below) 
  \[\mathcal{D}:=(\mathcal{H}_0\setminus C_{\alpha_0})\cup(\mathcal{H}_1\setminus C_{\alpha_1})\cup \st_2\cup \st_3\cup C_L\cup B_2\cup B_3\]whose number of components is
  \[(l(\mathcal{H}_0)-1)+(l(\mathcal{H}_1)-1)+l(\st_2)+l(\st_3)+1+1+1=b_2^-(X)+L.\] In any case, our choice of $B_2,B_3$ makes $\mathcal{D}$ a disjoint union of $C_L$ and several proper subtrees of reducible representatives of the fiber class $\tilde{F}_l$. Hence $\mathcal{D}$ is negative definite by Proposition \ref{prop:fiberclass}(2). However, this contradicts $b_2^-(X'\# L\overline{\CP}^2)=b_2^-(X)-1+L$.
\end{proof}

\begin{figure}[htbp]
  \centering
 
\begin{minipage}[c]{0.32\textwidth}
    \centering
    \resizebox{\linewidth}{0.74\linewidth}{%
    \begin{tikzpicture}[x=1cm,y=1cm,line cap=round,line join=round]
      \tikzset{
        div/.style={draw=black!60,dashed,line width=0.8pt},
        comp/.style={draw=black!65,line width=0.9pt},
        base/.style={draw=black!65,line width=0.8pt}
      }
      \node at (6.5,0.25) {$C_L$};
      \node at (2,5) {$\mathcal{H}_0$};
      \node at (5,5) {$\st_2$};
      \node at (8,5) {$\st_3$};
      \node at (11,5) {$\mathcal{H}_1$};
      \draw[base] (-0.3,0) -- (13.4,0);
      \fill (1.95,0) circle (2.2pt);
      \fill (11.2,0) circle (2.2pt);
      \fill (7.95,0) circle (2.2pt);
      \node at (1.9,-0.5) {$\Theta_0=\Theta_2$};
      \node at (7.9,-0.5) {$\Theta_3$};
      \node at (11.1,-0.5) {$\Theta_1$};

      \draw[div] (1.2,-1.15) rectangle (2.7,4.55);
      \node at (1.95,4.12) {$\cdots$};
      \node at (1.95,2.42) {$\cdots$};
      \node at (1.45,3.2) {$\tilde{A}_k'$};
      \node at (2.3,0.25) {$C_{\alpha_0}$};
      \draw[comp] (1.72,3.55) -- (2.25,2.72);
      \draw[comp] (1.72,2.08) -- (2.12,0.95);
      \draw[comp] (2.12,1.60) -- (1.58,0.32);
      \draw[comp] (1.66,0.72) -- (2.08,-0.35);

      \draw[div] (4.3,1.5) rectangle (5.8,4.65);
      \node at (5.05,1.7) {$\cdots$};
      \draw[comp] (4.78,4.15) -- (5.45,3.02);
      \draw[comp] (4.58,3.08) -- (5.28,1.98);
      \draw[comp] (5.33,3.48) -- (4.68,2.55);
      \draw[base] (2.7,3.30) -- (4.3,3.30);
      \node at (3.4,3.5) {$B_2$};

      \draw[div] (7.2,1.5) rectangle (8.7,4.65);
      \node at (7.95,1.7) {$\cdots$};
      \draw[comp] (7.95,1.5) -- (7.95,0);
      \draw[comp] (7.68,4.20) -- (8.33,2.95);
      \draw[comp] (7.50,3.02) -- (8.18,1.98);
      \draw[comp] (8.22,3.62) -- (7.62,2.60);
      \draw[comp] (8.7,3.3) -- (9.6,3.3);
      \node at (9.1,3.5) {$B_3$};

      \draw[div] (10.2,-1.15) rectangle (11.9,4.65);
      \node at (11.05,3.05) {$\cdots$};
      \node at (11.5,0.25) {$C_{\alpha_1}$};
      \draw[comp] (10.95,1.85) -- (11.55,0.72);
      \draw[comp] (11.40,1.38) -- (10.88,0.18);
      \draw[comp] (10.98,0.72) -- (11.30,-0.35);
    \end{tikzpicture}%
    }
  \end{minipage}\hfill
  \begin{minipage}[c]{0.32\textwidth}
    \centering
    \resizebox{\linewidth}{0.74\linewidth}{%
    \begin{tikzpicture}[x=1cm,y=1cm,line cap=round,line join=round]
      \tikzset{
        div/.style={draw=black!60,dashed,line width=0.8pt},
        comp/.style={draw=black!65,line width=0.9pt},
        base/.style={draw=black!65,line width=0.8pt}
      }
      \node at (6.5,0.25) {$C_L$};
      \node at (2,5) {$\mathcal{H}_0$};
      \node at (5,5) {$\st_2$};
      \node at (8,5) {$\st_3$};
      \node at (11,5) {$\mathcal{H}_1$};
      \draw[base] (-0.3,0) -- (13.4,0);
      \fill (1.95,0) circle (2.2pt);
      \fill (11.2,0) circle (2.2pt);
      \node at (1.9,-0.5) {$\Theta_0=\Theta_2=\Theta_3$};
      \node at (11.1,-0.5) {$\Theta_1$};

      \draw[div] (1.2,-1.15) rectangle (2.7,4.55);
      \node at (1.95,4.12) {$\cdots$};
      \node at (1.95,2.42) {$\cdots$};
      \node at (1.45,3.2) {$\tilde{A}_k'$};
      \node at (2.3,0.25) {$C_{\alpha_0}$};
      \draw[comp] (1.72,3.55) -- (2.25,2.72);
      \draw[comp] (1.72,2.08) -- (2.12,0.95);
      \draw[comp] (2.12,1.60) -- (1.58,0.32);
      \draw[comp] (1.66,0.72) -- (2.08,-0.35);

      \draw[div] (4.3,1.5) rectangle (5.8,4.65);
      \node at (5.05,1.7) {$\cdots$};
      \draw[comp] (4.78,4.15) -- (5.45,3.02);
      \draw[comp] (4.58,3.08) -- (5.28,1.98);
      \draw[comp] (5.33,3.48) -- (4.68,2.55);
      \draw[base] (2.7,3.30) -- (4.3,3.30);
      \node at (3.4,3.5) {$B_3$};

      \draw[div] (7.2,1.5) rectangle (8.7,4.65);
      \node at (7.95,1.7) {$\cdots$};
      \draw[comp] (7.68,4.20) -- (8.33,2.95);
      \draw[comp] (7.50,3.02) -- (8.18,1.98);
      \draw[comp] (8.22,3.62) -- (7.62,2.60);
      \draw[base] (5.8,2.55) -- (7.2,2.55);
      \node at (6.4,2.75) {$B_2$};

      \draw[div] (10.2,-1.15) rectangle (11.9,4.65);
      \node at (11.05,3.05) {$\cdots$};
      \node at (11.5,0.25) {$C_{\alpha_1}$};
      \draw[comp] (10.95,1.85) -- (11.55,0.72);
      \draw[comp] (11.40,1.38) -- (10.88,0.18);
      \draw[comp] (10.98,0.72) -- (11.30,-0.35);
    \end{tikzpicture}%
    }
  \end{minipage}\hfill
  \begin{minipage}[c]{0.32\textwidth}
    \centering
    \resizebox{\linewidth}{0.74\linewidth}{%
    \begin{tikzpicture}[x=1cm,y=1cm,line cap=round,line join=round]
      \tikzset{
        div/.style={draw=black!60,dashed,line width=0.8pt},
        comp/.style={draw=black!65,line width=0.9pt},
        base/.style={draw=black!65,line width=0.8pt}
      }
      \node at (6.5,0.25) {$C_L$};
      \node at (2,5) {$\mathcal{H}_0$};
      \node at (5,5) {$\st_2$};
      \node at (8,5) {$\st_3$};
      \node at (11,5) {$\mathcal{H}_1$};
      \draw[base] (-0.3,0) -- (13.4,0);
      \fill (1.95,0) circle (2.2pt);
      \fill (11.2,0) circle (2.2pt);
      \fill (7.95,0) circle (2.2pt);
      \node at (1.9,-0.5) {$\Theta_0$};
      \node at (7.9,-0.5) {$\Theta_2=\Theta_3$};
      \node at (11.1,-0.5) {$\Theta_1$};

      \draw[div] (1.2,-1.15) rectangle (2.7,4.55);
      \node at (1.95,4.12) {$\cdots$};
      \node at (1.95,2.42) {$\cdots$};
      \node at (1.45,3.2) {$\tilde{A}_k'$};
      \node at (2.3,0.25) {$C_{\alpha_0}$};
      \draw[comp] (1.72,3.55) -- (2.25,2.72);
      \draw[comp] (1.72,2.08) -- (2.12,0.95);
      \draw[comp] (2.12,1.60) -- (1.58,0.32);
      \draw[comp] (1.66,0.72) -- (2.08,-0.35);

      \draw[div] (4.3,1.5) rectangle (5.8,4.65);
      \node at (5.05,1.7) {$\cdots$};
      \draw[comp] (4.78,4.15) -- (5.45,3.02);
      \draw[comp] (4.58,3.08) -- (5.28,1.98);
      \draw[comp] (5.33,3.48) -- (4.68,2.55);
      \node at (6.4,2.75) {$B_2$};

      \draw[div] (7.2,1.5) rectangle (8.7,4.65);
      \node at (7.95,1.7) {$\cdots$};
      \draw[comp] (7.68,4.20) -- (8.33,2.95);
      \draw[comp] (7.50,3.02) -- (8.18,1.98);
      \draw[comp] (8.22,3.62) -- (7.62,2.60);
      \draw[base] (5.8,2.55) -- (7.2,2.55);
      \draw[comp] (7.95,1.5) -- (7.95,0);

      \draw[div] (10.2,-1.15) rectangle (11.9,4.65);
      \node at (11.05,3.05) {$\cdots$};
      \node at (11.5,0.25) {$C_{\alpha_1}$};
      \draw[comp] (10.95,1.85) -- (11.55,0.72);
      \draw[comp] (11.40,1.38) -- (10.88,0.18);
      \draw[comp] (10.98,0.72) -- (11.30,-0.35);
      \draw[comp] (11.9,3) -- (12.9,3);
      \node at (12.4,3.2) {$B_3$};
    \end{tikzpicture}%
    }
  \end{minipage}\hfill
  \caption{Schematic configurations used in the proof of Proposition \ref{prop:nonontoric}. The dashed boxes indicate the strings $\mathcal{H}_0$, $\mathcal{H}_1$, $\st_2$, and $\st_3$, while the solid segments represent the components of reducible representatives of the fiber class. The bottom horizontal segment denotes the $C_L$-component, which can be identified with the moduli space of curves in the fiber class by \cite{Zhangmoduli}. The dots on $C_L$ mark the elements $\Theta_0,\Theta_1,\Theta_2,\Theta_3$ in moduli space. Depending on which $\Theta_i$ coincide, we choose different $B_2,B_3$-components to form $\mathcal{D}$.}
  \label{fig:nonontoric}
\end{figure}
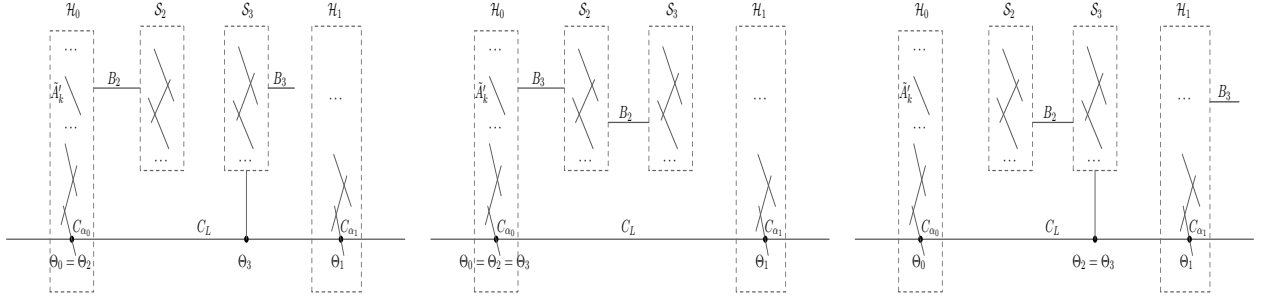

With the above understood, we now investigate the behaviour of the exceptional sphere of minimal symplectic area under the additional $(a,b,c)$-type and gap-admissible conditions.

\begin{proposition}\label{prop:HJandEmin}
  Assume that the disjoint union of three HJ strings $\st_1,\st_2,\st_3$ is full, of $(a,b,c)$-type for pairwise coprime $(a,b,c)$ and gap-admissible. Let $E_{\min}\in\mathcal{E}_{\omega}$ be an exceptional class of minimal symplectic area.
  Then
  \(
    E_{\min}\cdot\bigl([\st_1]+[\st_2]+[\st_3]\bigr)= 2,
  \)
  and for each $\alpha\in\{1,2,3\}$,
  \(
    0\le E_{\min}\cdot[\st_\alpha]\le 1.
  \)
\end{proposition}

\begin{proof}
  Choose any divisor-adapted almost complex structure $J\in\mathcal{J}(D)$. Since $E_{\text{min}}$ has an embedded $J$-holomorphic representative $C_{E_{\text{min}}}$ and $\st_1,\st_2,\st_3$ are Hirzebruch--Jung, positivity of intersection implies that 
  \(
    0\le E_{\min}\cdot[\st_\alpha]
  \)
  for each $\alpha\in\{1,2,3\}$. 
  
  If $E_{\min}\cdot[\st_\alpha]\geq 2$ for some $\alpha$, then the string $\st_\alpha$ is quasi-minimal (Definition \ref{def:quasimin}). Since $[\omega]\cdot([S_\alpha]+K_\omega)<0$, it follows from Proposition \ref{prop:-K-D=Emin} that $E_{\min}\cdot[\st_\alpha]= 2$ and $\st_\alpha$ must be quasi-minimal of first kind by the Hirzebruch--Jung assumption. Moreover, Proposition \ref{prop:twokinds} implies that $\st_\alpha\cup C_{E_\text{min}}$ is a cycle of spheres representing the class $-K_\omega$. By adjunction formula and positivity of intersection, for any component $C$ in the other two strings we have 
  \[-2\geq [C]^2=-2-K_\omega\cdot[C]=-2-([\st_\alpha]+E_{\text{min}})\cdot [C]\leq -2.\]
  Hence, the other two strings must be $(-2)$-strings. By Lemma \ref{lem:-2chains}, up to permutation of $(a,b,c)$, we may assume that two $(-2)$-strings have lengths $a-1$ and $b-1$ and the self-intersection sequence of $\st_\alpha$ is one of the following: $(-a,-k,-b)$ for some $k\geq 1$, $(1-a,1-b)$, or $(2-a)$. Then, we compute
  \[b_2^-(X)=9-K_{\omega}^2=9-([\st_\alpha]+E_{\text{min}})^2=6-[\st_\alpha]^2.\]
  For the case of $(-a,-k,-b)$, $[\st_\alpha]^2=-a-k-b+4$ and $\sum_{i=1}^3l(\st_i)=a+b+1$. The full condition would then imply that $k=-1$, which is a contradiction. Similarly, for the case of $(1-a,1-b)$ (resp. $(2-a)$), $[\st_\alpha]^2=-a-b+4$ (resp. $2-a$) and $\sum_{i=1}^3l(\st_i)=a+b$ (resp. $a+1$), which also contradicts with the full condition. Therefore, we have seen that $E_{\min}\cdot[\st_\alpha]\in\{0,1\}$. 

  By Proposition \ref{prop:nonontoric}, we cannot have
  \(
    E_{\min}\cdot\bigl([\st_1]+[\st_2]+[\st_3]\bigr)=1,
  \)
  since this would produce a non-toric exceptional sphere. It therefore remains to rule out the case
  \[
    E_{\min}\cdot[\st_1]=E_{\min}\cdot[\st_2]=E_{\min}\cdot[\st_3]=1.
  \]
  Suppose this occurs. Then, by the gap-admissible assumption, the divisor
  \[
    \overline{D}:=\st_1\cup\st_2\cup\st_3\cup C_{E_\text{min}}
  \]
  satisfies $[\omega]\cdot(K_\omega+[\overline{D}])<0$, and hence $(X,\omega,\overline{D})$ is a quasi-minimal pair of second kind. Since there are three strings, we know $b_2(X)\geq 4$ by the full assumption. Thus,  Proposition \ref{prop:twokinds} implies that there must be a toric/non-toric/half-toric/exterior exceptional sphere to the divisor $\overline{D}$. However, there could not be toric or half-toric ones since $\st_i$'s are Hirzebruch--Jung; there could not be non-toric ones since Proposition \ref{prop:nonontoric} implies that they must be disjoint from $\st_1\cup\st_2\cup\st_3$ which will contradict with Lemma \ref{lem:noexterior}; there could not be exterior ones directly by Lemma \ref{lem:noexterior}.
\end{proof}

\subsection{Connecting three HJ strings into one by exceptional spheres}\label{section:connecting}
Therefore, under the assumptions of Definition \ref{def:3assumptions}, we may assume that $C_{E_{\text{min}}}$ intersects $\st_1$ and $\st_2$ by Proposition \ref{prop:HJandEmin}. It follows that
\[
  \overline{D} = \st_{12} \cup \st_3,
  \qquad
  \st_{12} := \st_1  \cup C_{E_{\text{min}}}\cup \st_2,
\]
so $\overline{D}$ has exactly two connected components, namely $\st_{12}$ and $\st_3$. 
Our next goal is to connect $\st_3$ to either $\st_1$ or $\st_2$ by some exceptional sphere.

\begin{proposition}\label{prop:secondexceptional}
 Under the above assumptions, after possibly exchanging $\st_1$ and $\st_2$, there exists a symplectic exceptional sphere $C_e$ that intersects exactly one component of $\st_1$ and $\st_3$ transversely and positively in one point each, and is disjoint from $\st_2$.
\end{proposition}

\begin{proof}
  We argue by contradiction, assuming that neither $\st_1$ nor $\st_2$ can be connected to $\st_3$ by a symplectic exceptional sphere. After performing the toric blowdown of the $C_{E_{\text{min}}}$-component, we obtain a new pair $(X',\omega',D')$ where $D'=\st_{12}'\cup\st_3'$ such that $\st_{12}'=\st_1'\cup\st_2'$ and each $\st_i$ is the proper transform of $\st_i'$. Notice that the gap-admissible condition ensures $[\omega]\cdot(K_\omega+[\overline{D}])<0$, which also implies  $[\omega']\cdot(K_{\omega'}+[D'])<0$ after toric blowdown by Lemma \ref{lem:blowdownarea}.

 \begin{claim}
$(X',\omega',D')$ can be obtained by successive toric or half-toric blowups from a pair $(X'',\omega'',D'')$ such that one of the connected component of $D''$ is quasi-minimal.
\end{claim}
 \begin{proof}[Proof of claim:] Choose $J'\in\mathcal{J}(D')$ and let $C'$ be an embedded $J'$-holomorphic representative of a minimal symplectic exceptional class in $(X',\omega')$. We may assume
 \[
 [C']\cdot [\st_{12}']\leq 1,\qquad [C']\cdot [\st_3']\leq 1,
 \]
 since otherwise $\st_{12}'$ or $\st_3'$ is already quasi-minimal. If $C'$ is not a component of $D'$, then by Lemma \ref{lem:blowupexceptional}, which excludes exceptional spheres connecting $\st_{12}'$ and $\st_3'$ by our assumption, $C'$ must be disjoint from one of these two connected components. This would imply that $C'$ is either non-toric, contradicting Proposition \ref{prop:nonontoric} and Lemma \ref{lem:blowupexceptional}, or exterior, contradicting the full condition. If instead $C'$ is a component of $D'$, then it is toric or half-toric, so we can blow it down and obtain a new pair $(X'_1,\omega'_1,D'_1)$. The inequality
 \(
 [\omega'_1]\cdot(K_{\omega'_1}+[D'_1])<0
 \)
 still holds after this blowdown. Repeating the same argument, we conclude that either one connected component becomes quasi-minimal at some stage, or we can continue performing toric or half-toric blowdowns until $b_2^-=1$. Notice that each such blowdown can only occur along components in $\st_{12}$, since $\st_3$ is Hirzebruch--Jung, and each step decreases the number of divisor components by $1$. Hence, if $b_2^-$ were reduced to $1$, then necessarily $l(\st_3)=1$, and exactly two components would remain; these must be the dual sections in a Hirzebruch surface. Let $F$ denote the fiber class of that Hirzebruch surface. Note that the final blowdown from $b_2^-=2$ to $b_2^-=1$ must be a half-toric blowdown of an exceptional component in class $e$. It follows that, in the blowup of the Hirzebruch surface, there is an exceptional sphere in class $F-e$ connecting the two connected components of the divisor. This contradicts our assumption by Lemma \ref{lem:blowupexceptional}. Therefore, $b_2^-$ cannot be reduced to $1$, and the claim is proved.
\end{proof}
 With the claim understood, we may now assume $D''=D''_1\cup D''_2$ with $D''_1$ being quasi-minimal with $e''\in\mathcal{E}_{\text{min}}(X'',\omega'')$ such that $e''\cdot [D_1'']=2$.
 \begin{claim}
Given any $J''\in\mathcal{J}(D'')$, the embedded $J''$-holomorphic representative $C_{e''}$ of $e''$ must be disjoint from $D_2''$.
 \end{claim}
\begin{proof}
  Otherwise, by Lemma \ref{lem:blowupexceptional}, the class $e''\in H_2(X'';\mathbb{Z})\subseteq H_2(X;\mathbb{Z})$ would be an exceptional class in either $\mathcal{E}_\omega(\st_1,\st_3;D)$ or $\mathcal{E}_\omega(\st_2,\st_3;D)$. By Proposition \ref{prop:-K-D=Emin},\[[\omega'']\cdot(K_{\omega''}+[D''])=[\omega'']\cdot(K_{\omega''}+[D'']+e''- e'')=[\omega'']\cdot[D_2'']-[\omega'']\cdot e''>-[\omega'']\cdot e''.\]
  However, by Lemma \ref{lem:blowdownarea}, we then have
  \begin{align*}
  [\omega]\cdot(K_{\omega}+[D])&=[\omega']\cdot(K_{\omega'}+[D'])-[\omega]\cdot E_{\text{min}}\\
  &\geq [\omega'']\cdot(K_{\omega''}+[D''])-[\omega]\cdot E_{\text{min}}\\
  &>-[\omega'']\cdot e''-[\omega]\cdot E_{\text{min}}=-[\omega]\cdot e''-[\omega]\cdot E_{\text{min}},
  \end{align*}
  which contradicts the gap-admissible assumption.
\end{proof}
Hence, by adjunction formula, components in $D_2''$ can only be $(-2)$-spheres since $e''+[D_1'']=-K_{\omega''}$. To see the contradiction, let us introduce the quantity $\Xi(\mathcal{C})$ for any symplectic divisor $\mathcal{C}=C_1\cup\cdots \cup C_l$
\[\Xi(\mathcal{C}):=-3l-\sum_{i=1}^l[C_i]^2.\]
By Corollary \ref{cor:3nminus6}, we have $\Xi(D)\geq -6$, hence $\Xi(D')\geq -4$ by direct computation. Moreover, each toric blowdown leaves $\Xi$ unchanged, while each half-toric blowdown increases $\Xi$ by $1$. Therefore, after passing from $D'$ to $D''$ through a sequence of toric or half-toric blowdowns, we obtain $\Xi(D'')\geq -4$. Since $D_1''$ is quasi-minimal, $-K_{\omega''}=e''+[D_1'']$ so that 
\[K_{\omega''}^2=(e''+[D_1''])^2=-1+[D_1'']^2+2e''\cdot[D_1'']=3+[D_1'']^2=3+\sum_{i=1}^{l(D_1'')} [C_i'']^2+2l(D_1'')-2,\]
where $C_i''$ denotes the component in $D_1''$ and $l(D_1'')$ denotes the number of components. Note that the last equality holds true since the configuration of $D_1''$ is a tree with no loop. So,
\[\Xi(D_1'')=-3l(D_1'')-\sum_{i=1}^{l(D_1'')} [C_i'']^2=-3l(D_1'')-\big(K_{\omega''}^2-2l(D_1'')-1\big)=-l(D_1'')-K_{\omega''}^2+1.\]
For $D_2''$, since it only contains $(-2)$-spheres, we see 
\[\Xi(D_2'')=-3l(D_2'')+2l(D_2'')=-l(D_2'').\]
 Observe that $l(D_1'')+l(D_2'')=b_2^-(X'')+1$ since only toric or half-toric blowdowns were performed. Therefore, we compute that 
\begin{align*}
\Xi(D'')&=\Xi(D_1'')+\Xi(D_2'')=-l(D_1'')-l(D_2'')-K_{\omega''}^2+1\\
&=-b_2^-(X'')-K_{\omega''}^2=-b_2^-(X'')-(9-b_2^-(X''))=-9,
\end{align*}
which contradicts $\Xi(D'')\geq -4$.
\end{proof}

In terms of the notion of exceptional gaps, the above proposition particularly implies that $\mathfrak{e}(\st_{12},\st_3;\overline{D})>0$. As a result, we obtain a connected symplectic divisor
\[\overline{\overline{D}}:=\overline{D}\cup C_e,\]
where, after possibly exchanging the index of $\st_1$ and $\st_2$, we may assume that $C_e$ is an exceptional sphere connecting $\st_1$ and $\st_3$. For any symplectic divisor $\mathcal{C}=\bigcup_i C_i$, we call $C_i$ a \emph{branching component} if
\[
[C_i]\cdot([\mathcal{C}]-[C_i])\geq 3.
\]
Our aim is to show that $\overline{\overline{D}}$ cannot have branching component.

\begin{lemma}\label{lem:-1reduction}
  If $\overline{\overline{D}}$ has a branching component $C$ with $[C]\cdot([\overline{\overline{D}}]-[C])= 3$ (resp. $4$), then one (resp. two) of the connected components of $\overline{\overline{D}}\setminus C$ can be reduced to a $(-1)$-sphere by a sequence of toric or half-toric blowdowns.
\end{lemma}
\begin{proof}
  By the gap-admissible assumption, $\overline{\overline{D}}$ is a connected divisor such that $[\omega]\cdot(K_{\omega}+[\overline{\overline{D}}])<0$. By Proposition \ref{prop:minimalreduction}, it can be reduced to a pair $(\check{X},\check{\omega},\check{D})$ which is either quasi-minimal or $b_2(\check{X})\leq 2$ by successive blowdowns. By Lemma \ref{lem:blowupexceptional}, if a non-toric or exterior blowdown was performed at some step, then $(X,\omega,D)$ will admit some non-toric or exterior exceptional sphere (if it is non-toric to $\overline{\overline{D}}$ which intersects $C_e$ or $C_{E_{\text{min}}}$, then it is exterior to $D$), which will violate either Proposition \ref{prop:nonontoric} or the full condition. Hence, the reduction process only involves toric or half-toric blowdowns. Since $l(\overline{\overline{D}})=2+b_2^-(X)=1+b_2(X)$ and each toric or half-toric blowdowns decreases the number of divisor components by $1$, we have $l(\check{D})=1+b_2(\check{X})$.
  
  Note that if $\check{D}$ has no branching component, then in the process of successive blowdowns starting from $\overline{\overline{D}}$, there must be a step at which a half-toric blowdown is performed at each branching component, so that the resulting divisor $\check{D}$ has no branching component. This implies that one (if $\overline{\overline{D}}\setminus C$ has three connected components) or two (if $\overline{\overline{D}}\setminus C$ has four connected components) of the connected component of $\overline{\overline{D}}\setminus C$ can be contracted to a $(-1)$-sphere. Therefore, it remains to discuss the following cases to show that $\check{D}$ cannot have a branching component.
  \begin{itemize}[nosep]
    \item If $b_2(\check{X})=2$, then $\check{D}$ has no branching component since it can only have three components. 
    \item If $(\check{X},\check{\omega},\check{D})$ is quasi-minimal of the first kind, then it must be a string by Proposition \ref{prop:twokinds}, so no branching component can appear.
    \item If $(\check{X},\check{\omega},\check{D})$ is quasi-minimal of the second kind, then we can further reduce it to a divisor in a manifold with $b_2\leq 2$ by Proposition \ref{prop:twokinds}. By the same reason as above, the reduction can only involve toric or half-toric blowdowns. Thus, we are back in the case of the first item.

  \end{itemize}

\end{proof}

\begin{remark}\label{rmk:-1reduction}
  The argument above yields the following more general statement: if $D\subseteq (X,\omega)$ is a connected symplectic divisor with $[\omega]\cdot (K_\omega+[D])<0$ and $l(D)\leq b_2(X)+1$, and if $D$ admits no non-toric or exterior exceptional spheres, then one of the connected components of the complement of its branching component can be contracted to a $(-1)$-sphere.
\end{remark}

\begin{proposition}\label{prop:isstring}
$\overline{\overline{D}}$ is a string.
\end{proposition}

\begin{proof}
  Recall that we have assumed $\st_2,\st_3$ are connected to $\st_1$ by $C_{E_{\text{min}}},C_e$ respectively. Let us first show the branching component cannot appear on $\st_2$ or $\st_3$. Otherwise, the complement of the branching component will have exactly three connected components $\mathcal{C}_1,\mathcal{C}_2,\mathcal{C}_3$, where two of them, say $\mathcal{C}_1,\mathcal{C}_2$, are HJ strings since they are part of $\st_2$ (resp. $\st_3$). Since HJ strings contain no $(-1)$-sphere, we know that the other connected component $\mathcal{C}_3=C_{E_\text{min}}\cup \st_1\cup C_e\cup \st_3$ (resp. $\mathcal{C}_3=C_e\cup \st_1\cup C_{E_\text{min}}\cup \st_2$) can be reduced to a $(-1)$-sphere by successive toric or half-toric blowdowns by Lemma \ref{lem:-1reduction}. Therefore, if we futher perform one half-toric blowdown of that $(-1)$-sphere, we will obtain a string with the same length as $\st_2$ (resp. $\st_3$) in a manifold with \[b_2^-=b_2^-(X)-l(\mathcal{C}_3)=b_2^-(X)-l(\st_1)-l(\st_3)-2=l(\st_2)-2\]
  \[(\text{resp. } b_2^-=b_2^-(X)-l(\mathcal{C}_3)=b_2^-(X)-l(\st_1)-l(\st_2)-2=l(\st_3)-2).\] However, $\mathcal{C}_1\cup \mathcal{C}_2$ will descend and become a negative definite configuration with $l(\st_2)-1$ (resp. $l(\st_3)-1$) components, which is contradiction.

  Next, we show $\st_1$ does not contain branching component neither. Assume the there was a branching component $C\subseteq \st_1$, then we claim the following.
  \begin{claim}
$\st_2,\st_3$ must be $(-2)$-strings.
\end{claim}

\begin{proof}
  Suppose first that $C$ is attached to both $\st_2$ and $\st_3$ through $C_{E_{\text{min}}}$ and $C_e$ respectively. Then $\overline{\overline{D}}\setminus C$ has either $3$ or $4$ connected components, depending on whether $C$ is an endpoint of $\st_1$ or not. Hence $\overline{\overline{D}}\setminus C$ is negative definite: one or two components are HJ strings, and the remaining two are the strings $C_{E_{\text{min}}}\cup \st_2$ and $C_e\cup \st_3$, which are also negative definite by Lemma \ref{lem:-1endpoint}. This is a contradiction since $\overline{\overline{D}}\setminus C$ has $b_2^-(X)+1$ components. 
  
  Now suppose instead that only $C_{E_{\text{min}}}\cup \st_2$ (resp. $C_e\cup\st_3$) is attached to $C$. Write the three connected components of $\overline{\overline{D}}\setminus C$ as $\mathcal{C}_1,\mathcal{C}_2,\mathcal{C}_3$, where $\mathcal{C}_1\subseteq \st_1$, $\mathcal{C}_2$ contains $\st_3$ (resp. $\st_2$), and $\mathcal{C}_3=C_{E_{\text{min}}}\cup \st_2$ (resp. $\mathcal{C}_3=C_e\cup \st_3$). Since $\mathcal{C}_1$ is an HJ string, it cannot be contracted to a $(-1)$-sphere. The same is true for $\mathcal{C}_2$: otherwise, after contracting $\mathcal{C}_2$, the image of $\mathcal{C}_1\cup\mathcal{C}_3$ would be a negative definite configuration with components more than $b_2^-$ of the blowdown manifold, by the same argument as above. Therefore, Lemma \ref{lem:-1reduction} implies that $\mathcal{C}_3$ can be contracted to a $(-1)$-sphere, and thus $\st_2$ (resp. $\st_3$) must be a $(-2)$-string.
\end{proof}
The above claim allows us to apply Lemma \ref{lem:-2chains}. The case $k=1$ in Lemma \ref{lem:-2chains} cannot occur: if a branching component existed in that case, then $C_{E_{\text{min}}}\cup \st_2$ and $C_e\cup \st_3$ would have to attach to the same component of $\st_1$, and deleting that branching component from $\overline{\overline{D}}$ would yield a negative definite configuration with $b_2^-(X)+1$ components. Therefore $k\geq 2$, and $\st_1$ is a three-component string with self-intersection sequence $(-a,-k,-b)$. Since $\st_2$ and $\st_3$ contain no branching components, any branching component must lie in $\st_1$. For the same reason, $C_{E_{\text{min}}}\cup \st_2$ and $C_e\cup \st_3$ cannot be attached to the same component of $\st_1$.
After contracting both $C_{E_{\text{min}}}\cup \st_2$ and $C_e\cup \st_3$, $\st_1$ will descend to a three-component string $\st_1'$ in a Hirzebruch surface as $b_2$ is always two less than the number of divisor components.
Moreover, one endpoint of the string $\st_1'$, whose proper transform is the component in $\st_1$ where neither $C_{E_{\text{min}}}$ nor $C_e$ is attached, will have self-intersection $-a$ or $-b$.

However, such a configuration $\st_1'$ cannot exist in a Hirzebruch surface by elementary homological considerations. We use the standard bases $F,B\in H_2(S^2\times S^2;\mathbb{Z})$, where $F$ and $B$ denote the classes of the two $S^2$-factors, and $H,E\in H_2(\CP^2\#\overline{\CP}^2;\mathbb{Z})$, where $H$ and $E$ denote the line class and the exceptional class, respectively. The following claim is well-known, due to the uniqueness result of symplectic structures on ruled surfaces (\cite{lalondemcduff}).
\begin{claim}\label{claim:sphereinHirzebruch}
The classes of embedded symplectic spheres in $S^2\times S^2$ (resp. $\CP^2\#\overline{\CP}^2$) are precisely those satisfying the adjunction equality for spheres, which are of the form $F+sB$ or $B+tF$ for some $s,t\in\mathbb{Z}$ (resp. $(r+1)H-rE$ for some $r\in\mathbb{Z}$), and has positive symplectic area.
\end{claim}
Therefore, if $\st_1'$ has an endpoint with self-intersection $-a$ (resp. $-b$), then its other endpoint must have self-intersection $a$ (resp. $b$).
Since $l(C_{E_{\text{min}}}\cup \st_2)+l(C_e\cup \st_3)=a+b$, it follows that both $C_{E_{\text{min}}}\cup \st_2$ and $C_e\cup \st_3$ must attach to the $(-b)$ (resp. $(-a)$) component of $\st_1$, contradicting the conclusion above.
\end{proof}

\subsection{Completing the string into a cycle}\label{section:complete}
To finish the proof of Theorem \ref{thm:main}, it remains to construct a symplectic sphere in class $-K_\omega-[\overline{\overline{D}}]$ which completes the string $\overline{\overline{D}}$ into the boundary divisor of some toric moment map.

\begin{proposition}\label{prop:cycle}
  $-K_\omega-[\overline{\overline{D}}]$ can be represented by an embedded symplectic sphere $C$ such that $C\cup \overline{\overline{D}}$ is a toric symplectic log Calabi--Yau divisor.
\end{proposition}
\begin{proof}
  By the gap-admissible condition, $\overline{\overline{D}}\subseteq (X,\omega)$ is a connected symplectic divisor satisfying $[\omega]\cdot(K_\omega+[\overline{\overline{D}}])<0$. As in the proof of Lemma \ref{lem:-1reduction}, we can perform successive toric or half-toric blowdowns to obtain a pair $(\check{X},\check{\omega},\check{D})$ that is either quasi-minimal or satisfies $b_2(\check{X})=2$, by Proposition \ref{prop:minimalreduction}. Let us first show the following.
  \begin{claim}
   $-K_{\check{\omega}}-[\check{D}]$ can be represented by an embedded $\check{\omega}$-symplectic sphere $\check{C}$.
  \end{claim}
  \begin{proof}
     If $(\check{X},\check{\omega},\check{D})$ is quasi-minimal, then Proposition \ref{prop:-K-D=Emin} implies that $-K_{\check{\omega}}-[\check{D}]$ is a $\check{\omega}$-symplectic exceptional class of minimal area, and therefore it admits an embedded symplectic representative.

     It remains to consider the case $b_2(\check{X})=2$.  As $\check{D}$ is a chain of spheres, we compute
     \[
     (-K_{\check{\omega}}-[\check{D}])^2+K_{\check{\omega}}\cdot(-K_{\check{\omega}}-[\check{D}])
     =[\check{D}]^2+K_{\check{\omega}}\cdot[\check{D}]=-2.
     \]
     Hence $-K_{\check{\omega}}-[\check{D}]$ satisfies the adjunction equality for spheres. By Lemma \ref{lem:blowdownarea}, this class also has positive $\check{\omega}$-area, because $-K_\omega-[\overline{\overline{D}}]$ does. Therefore, Claim \ref{claim:sphereinHirzebruch} yields an embedded symplectic sphere representing $-K_{\check{\omega}}-[\check{D}]$.
  \end{proof}
Observe that
\[
  -K_\omega-[\overline{\overline{D}}]
  =-K_{\check{\omega}}-[\check{D}]-\sum E_i,
\]
where the $E_i$ are the exceptional classes corresponding to the half-toric blowdowns in the reduction process. By choosing each Darboux ball for the blowups sufficiently small and isotoping it to be relative to $\check{C}$, one obtains an embedded $\tilde{\omega}$-symplectic sphere $\tilde{C}$ representing $-K_\omega-[\overline{\overline{D}}]$, for some symplectic form $\tilde{\omega}$ deformation equivalent to $\omega$. Then, by \cite[Theorem 1.3]{DLW}, the class $-K_\omega-[\overline{\overline{D}}]$ also admits an embedded $\omega$-symplectic sphere smoothly isotopic to $\tilde{C}$. Also, notice that by Corollary \ref{cor:3nminus6}
\begin{align*}
(-K_\omega-[\overline{\overline{D}}])^2&=K_{\omega}^2+2K_\omega\cdot [\overline{\overline{D}}]+[\overline{\overline{D}}]^2=K_{\omega}^2+K_\omega\cdot [\overline{\overline{D}}]-2\\
&=9-b_2^-(X)+\sum_{C_i\subseteq \overline{\overline{D}}} (-2-[C_i]^2)-2\\
&=5-3b_2^-(X)-\sum_{C_i\subseteq D} [C_i]^2\geq -1.
\end{align*}
Moreover, the class $-K_\omega-[\overline{\overline{D}}]$ is primitive, since it has intersection number $1$ with the two endpoints of the string $\overline{\overline{D}}$. Its intersection number with every other component of $\overline{\overline{D}}$ is $0$.
Therefore, by \cite[Lemma 2.13]{LN25}, $-K_\omega-[\overline{\overline{D}}]$ is $\overline{\overline{D}}$-good and we can apply Theorem \ref{thm:MO15} to represent $-K_\omega-[\overline{\overline{D}}]$ by $C$ such that $C\cup \overline{\overline{D}}$ is a symplectic divisor. Since it represents the class $-K_\omega$ and has $b_2^-(X)+3$ components, it must be a toric symplectic log Calabi--Yau divisor.
\end{proof}

\begin{proof}[Proof of `if' part of Theorem \ref{thm:main}]
  By \cite[Theorem 1.5]{LMN22}, every toric symplectic log Calabi--Yau divisor can be realized as the boundary divisor of the moment map for a toric structure. Proposition \ref{prop:cycle} therefore implies that $D$ is sub-toric in the sense of Definition \ref{def:3assumptions}.
\end{proof}

With our main Theorem \ref{thm:main} established, we can now extend the Torelli theorem for symplectic log Calabi--Yau divisors to those arising from minimal resolutions of weighted projective planes.

\begin{proof}[Proof of Corollary \ref{cor:torelli}]
First, by \cite[Theorem 2.8]{LiLiAsian}, the integral isometry $\gamma$ on second (co)homology is realized by a diffeomorphism between $X$ and $X'$, since we assume $[\omega']\cdot K_{\omega'}<0$ and the sub-toric assumption on $D$ also implies $[\omega]\cdot K_\omega<0$. Moreover, cohomologous symplectic forms on a $4$-manifold have the same first Chern class by \cite[Corollary A]{Salamon}, so $\gamma$ must also identify $K_\omega$ with $K_{\omega'}$. Then \cite[Theorem A]{LiLiu96} implies that $\gamma$ identifies the sets of symplectic exceptional classes $\mathcal{E}_\omega$ and $\mathcal{E}_{\omega'}$. By definition, it therefore also identifies the sets of connecting log exceptional classes $\mathcal{E}_\omega^{\text{log}}(\st_i,\st_j;D)$ and $\mathcal{E}_{\omega'}^{\text{log}}(\st_i',\st_j';D')$. Since Theorem \ref{thm:main} shows that the sub-toric divisor $D$ is gap-admissible, the same holds for $D'$ by these identifications on the (co)homological level. Applying Theorem \ref{thm:main} once more, we conclude that $D'$ is also sub-toric, proving the first item of Corollary \ref{cor:torelli}.

Now assume further that $\gamma$ preserves the homology classes of all components of $D$ and $D'$. By Theorem \ref{thm:main}, both $D$ and $D'$ can be completed to toric symplectic log Calabi--Yau divisors by adding three components. Proposition \ref{prop:onlyif} shows that two of these added components belong to the set of connecting log exceptional classes, and these classes are uniquely determined by the total homology classes of the strings. Hence $\gamma$ preserves the classes of $l(D)+2=l(D')+2$ components of the toric log Calabi--Yau divisors. Since $\gamma$ also sends $K_\omega$ to $K_{\omega'}$, it preserves the class of the remaining component as well. We may therefore apply Theorem \ref{thm:LCYtorelli} to obtain a symplectomorphism between the corresponding log Calabi--Yau pairs. Restricting this symplectomorphism to $D$ and $D'$, we conclude that $(X,\omega,D)$ and $(X',\omega',D')$ are symplectomorphic. This proves the second item.

\end{proof}

\begin{remark}
  The intuition behind this argument is that certain symplectic divisors that arise as part of symplectic log Calabi--Yau divisors always exhibit this rigidity property. We expect that such a Torelli-type result also holds for more general divisors with symplectic log Kodaira dimension $-\infty$.
\end{remark}

\section{Symplectic affine rulings on weighted projective planes}\label{section:ruling}

In this section, we prove Theorem \ref{thm:ruling} by first constructing symplectic affine rulings on minimal resolutions of weighted projective planes and then showing how to recover the associated sub-toric configurations from symplectic affine rulings.

\subsection{Constructing affine rulings}
 We first consider the sub-toric configuration $D=\st_a\cup\st_b\cup\st_c\subseteq (X,\omega)$ arising from the symplectic minimal resolution of $\CP(a,b,c)$, where $a<b<c$ and $\st_a$, $\st_b$, and $\st_c$ denote the resolution strings at the orbifold points of orders $a$, $b$, and $c$, respectively. By Lemma \ref{lem:threeedges}, there exist exceptional spheres $N_a$ and $N_b$ connecting $\st_b$ to $\st_c$ and $\st_a$ to $\st_c$, respectively, as well as an embedded symplectic sphere $N_c$ connecting $\st_a$ to $\st_b$ with $[N_c]^2\geq -1$, such that $D\cup N_a\cup N_b\cup N_c$ is the toric boundary divisor of a moment map.
Consider the strings (with orientations indicated by labelings)
\begin{align*}
\st_{ac}:=\st_a\cup N_b\cup\st_c=A_1\cup\cdots\cup A_{k_a}\cup N_b\cup C_1\cup\cdots\cup C_{k_c},\\
\st_{bc}:=\st_b\cup N_a\cup\st_c=B_1\cup\cdots\cup B_{k_b}\cup N_a\cup C_{k_c}\cup\cdots\cup C_{1}
\end{align*} which are not negative definite by the full condition. Given any $\nu\in\{1,2,\cdots,k_c\}$, define their truncations by
\begin{align*}
\st_{ac}^{\leq\nu}:=\st_a\cup N_b\cup\st_c=A_1\cup\cdots\cup A_{k_a}\cup N_b\cup C_1\cup\cdots\cup C_{\nu},\\
\st_{bc}^{\geq\nu}:=\st_b\cup N_a\cup\st_c=B_1\cup\cdots\cup B_{k_b}\cup N_a\cup C_{k_c}\cup\cdots\cup C_{\nu},
\end{align*}and further define $\nu_a,\nu_b\in\{1,2,\cdots,k_c\}$ by
\begin{align*}
\nu_a:=\inf\{\nu\,|\,\st_{ac}^{\leq\nu} \text{ is not negative definite}\},\\
\nu_b:=\sup\{\nu\,|\,\st_{bc}^{\geq\nu} \text{ is not negative definite}\}.
\end{align*}
By the discussion in Section \ref{section:resfiberclass}, there exist two non-zero classes $F_a$ and $F_b$, obtained as positive integral linear combinations of the classes in $\st_{ac}^{\leq\nu_a}$ and $\st_{bc}^{\geq\nu_b}$, respectively, such that for any component $S$ in $D\cup N_a\cup N_b\cup N_c$
\[ F_a\cdot [S]=\begin{cases}
  p_a \geq 0,\qquad & S=C_{\nu_a}\\
  q_a >0 ,\qquad & S=C_{\nu_a+1}\\
  1,\qquad & S=N_c\\
  0,\qquad & \text{otherwise,}\end{cases}\qquad 
  F_b\cdot [S]=\begin{cases}
  p_b \geq 0,\qquad & S=C_{\nu_b}\\
  q_b > 0 ,\qquad & S=C_{\nu_b-1}\\
  1,\qquad & S=N_c\\
  0,\qquad & \text{otherwise,}
\end{cases}\]
where $(p_a,q_a)$ and $(p_b,q_b)$ are two pairs of coprime integers. Also, we have \[F_a^2=p_aq_a\geq 0,\quad F_b^2=p_bq_b\geq 0, \quad K_\omega\cdot F_a=-p_a-q_a-1,\quad K_\omega\cdot F_b=-p_b-q_b-1.\]


\begin{proposition}\label{prop:affineruling}
Under the above assumptions, the classes $F_a$ and $F_b$ coincide and one of the following holds.
\begin{enumerate}[nosep]
\item If $[N_c]^2\geq 0$, then $\nu_b-\nu_a=2$, and the class $F_a=F_b$ is represented by an embedded symplectic sphere of self-intersection $0$ that meets the unique component $C_{\nu_a+1}=C_{\nu_b-1}$ of $\st_c$ transversely in a single point and is disjoint from $\st_a\cup\st_b$.
\item If $[N_c]^2=-1$, then $\nu_b-\nu_a=1$, and the class $F_a=F_b$ is represented by a symplectic unicuspidal rational curve of self-intersection $p_aq_a=p_bq_b$ with a unique cusp singularity of type $(p_a,q_a)=(q_b,p_b)$ at the point $C_{\nu_a}\cap C_{\nu_b}$; away from this cusp, the curve is disjoint from $D$.
\end{enumerate}
\end{proposition}

\begin{proof}
{\bf Step 1:} $\nu_b-\nu_a\in\{1,2\}$.

When $\nu_b-\nu_a\geq 2$, the supports of the linear combinations defining $F_a$ and $F_b$ are disjoint, and hence $F_a\cdot F_b=0$. By the light cone lemma, it follows that $F_a$ and $F_b$ are proportional. However, if $\nu_b-\nu_a\geq 3$, this is impossible since $F_a\cdot [C_{\nu_b-2}]=0$ while $ F_b\cdot [C_{\nu_b-2}]>0$. Thus, $\nu_b-\nu_a\leq 2$.

If $\nu_b-\nu_a\leq 0$, then the configuration $(D\cup N_a\cup N_b)\setminus C_{\nu_a}$ splits as the disjoint union of two negative definite strings, namely $\st_{ac}^{\leq \nu_a-1}$ and $\st_{bc}^{\geq \nu_a+1}$. Together they contain $b_2^-(X)+1$ components, which is impossible. Therefore, $1\leq \nu_b-\nu_a$.

\medskip
 {\bf Step 2:} If $\nu_b-\nu_a=2$, then $F_a=F_b$ and $(p_a,q_a)=(p_b,q_b)=(0,1)$.

 We have seen in Step 1 that $\nu_b-\nu_a=2$ implies $F_a\cdot F_b=0$. Since also $F_a\cdot [N_c]=F_b\cdot [N_c]=1$, the light cone lemma implies that $F_a=F_b$.
  It follows that
  \[
    p_a=F_a\cdot [C_{\nu_a}]=F_b\cdot [C_{\nu_a}]=0,
    \qquad
    p_b=F_b\cdot [C_{\nu_b}]=F_a\cdot [C_{\nu_b}]=0,
  \]
  and hence $F_a^2=F_b^2=0$. Since $(p_a,q_a)$ and $(p_b,q_b)$ are coprime pairs, we must have $q_a=q_b=1$.
  
\medskip
{\bf Step 3:} If $\nu_b-\nu_a=1$, then $F_a=F_b$ and $(p_a,q_a)=(q_b,p_b)$.

We perform successive toric blowups at the node $C_{\nu_a}\cap C_{\nu_b}$, using the weight sequence $(m_1,\dots,m_l)$ associated to the pair $(p_a,q_a)$ as in Section \ref{section:resfiberclass}. This produces a new string $\tilde{\st}$ in the blowup manifold $(\tilde{X},\tilde{\omega})$ with components
\[A_1\cup\cdots\cup A_{k_a}\cup N_b\cup C_1\cup\cdots\cup \tilde{C}_{\nu_a}\cup \mathcal{L}\cup\tilde{C}_{\nu_b}\cup\cdots\cup C_{k_c}\cup N_a\cup B_{k_b}\cup\cdots\cup B_1,\]
where $\tilde{C}_{\nu_a}$ and $\tilde{C}_{\nu_b}$ are the proper transforms of $C_{\nu_a}$ and $C_{\nu_b}$, and $\mathcal{L}$ is the string created by the toric blowup process. Let $e_1,\dots,e_l$ denote the exceptional classes of these blowups. Then $\mathcal{L}$ has the form
\[\tilde{C}_{e_\alpha}\cup\cdots\cup\tilde{C}_{e_\beta}\cup C_{e_l}\cup\tilde{C}_{e_\gamma}\cup\cdots\cup\tilde{C}_{e_\theta},\]
where each $\tilde{C}_{e_i}$ is the proper transform of the exceptional sphere in class $e_i$, and $C_{e_l}$ is the exceptional sphere arising from the final blowup. The resolution class
\(\tilde{F}_a:=F_a-\sum_{i=1}^l m_i e_i\)
has intersection number $1$ with $C_{e_l}$ and intersection number $0$ with every other component of $\tilde{\st}$. It also satisfies $\tilde{F}_a\cdot [N_c]=1$. Now consider the truncated string $\tilde{\st}^{\geq \gamma}\subseteq \tilde{\st}$ with components
\[B_1\cup\cdots\cup B_{k_b}\cup N_a\cup C_{k_c}\cup\cdots\cup \tilde{C}_{\nu_b}\cup \tilde{C}_{e_\theta}\cup\cdots\cup \tilde{C}_{e_\gamma}.\]
This string cannot be negative definite. Indeed, otherwise $\tilde{\st}\setminus (\tilde{C}_{e_\beta}\cup C_{e_l})$ would be a negative definite configuration spanning a lattice of rank $b_2^-(X)+l$ inside $\langle\tilde{F}_a\rangle^\perp\subseteq H_2(\tilde{X};\mathbb{Z})$, which is impossible by Lemma \ref{lem:fibercomplement}. By the same argument as in Step 1, the truncated string $\tilde{\st}^{\geq \gamma}\setminus \tilde{C}_{e_\gamma}$ must nevertheless be negative definite. Therefore, as in Section \ref{section:resfiberclass}, if $\Delta_i$ denotes the determinant of the negative intersection matrix of the first $(i-1)$ components of $\tilde{\st}^{\geq \gamma}$, then the class
\[\tilde{F}_b:=\Delta_1[B_1]+\cdots+\Delta_{k_b}[B_{k_b}]+\Delta_{k_b+1}[N_a]+\Delta_{k_b+2}[C_{k_c}]+\cdots+\Delta_{l(\tilde{\st}^{\geq \gamma})}[\tilde{C}_{e_\gamma}]\]
satisfies
\(\tilde{F}_b^2=-\Delta_{l(\tilde{\st}^{\geq \gamma})}\Delta_{l(\tilde{\st}^{\geq \gamma})+1}\geq 0.\)
By the same argument as in Step 2, we conclude that $\tilde{F}_b=\tilde{F}_a$.
Since the self-intersection sequence of $\st_{bc}^{\geq \nu_b}\setminus C_{\nu_b}$ is unchanged by the toric blowups, we have
\[F_b=\Delta_1[B_1]+\cdots+\Delta_{k_b}[B_{k_b}]+\Delta_{k_b+1}[N_a]+\Delta_{k_b+2}[C_{k_c}]+\cdots+\Delta_{k_b+k_c-\nu_b+2}[C_{\nu_b}].\]
Hence
\[\tilde{F}_b-F_b=\Delta_{k_b+k_c-\nu_b+2}([\tilde{C}_{\nu_b}]-[C_{\nu_b}]) +\Delta_{k_b+k_c-\nu_b+3}[\tilde{C}_{e_\theta}]+\cdots+\Delta_{l(\tilde{\st}^{\geq \gamma})}[\tilde{C}_{e_\gamma}]\in \Z\langle e_1,\dots,e_l\rangle.\]
Therefore, the equality $\tilde{F}_a=\tilde{F}_b$ implies $F_a=F_b$, and the equality $(p_a,q_a)=(q_b,p_b)$ follows immediately.

\medskip
{\bf Step 4:} If $[N_c]^2\geq 0$, then $\nu_b-\nu_a=2$.

Suppose $\nu_b-\nu_a=1$. By Step 3, we have $F_b\cdot[C_{\nu_b}]=p_b=q_a>0$. Since the string $\st_{ac}^{\leq \nu_a}$ is non-negative definite, so is any of its toric or half-toric blowdowns. After blowing down the $N_c$-sphere, there must be at least one $(-1)$-component. Moreover, because $\st_{ac}^{\leq \nu_a}$ has only one $(-1)$-component, namely $N_c$, Lemma \ref{lem:adjacent-1} shows that one of the following cases occurs.
\begin{itemize}[nosep]
\item If there is only one $(-1)$-component after blowing down $N_c$, then by Lemma \ref{lem:-1endpoint} it cannot appear as an endpoint of the string. Hence we can perform a further toric blowdown.
\item If there are two adjacent $(-1)$-components after blowing down $N_c$, let $U$ be the sum of their homology classes, so that $U^2=0$. Since $[N_c]^2\geq 0$, if $U\cdot [N_c]=0$, then the light cone lemma implies that the nonzero classes $U$ and $[N_c]$ are proportional. This is impossible because $U\cdot [B_1]=0$ while $[N_c]\cdot [B_1]=1$. Thus one of the $(-1)$-components must intersect $N_c$. Moreover, if $U\cdot [C_{\nu_b}]=0$, then $U\cdot F_b=0$, so the nonzero classes $U$ and $F_b$ are again proportional by the light cone lemma. This contradicts $F_b\cdot [C_{\nu_b}]>0$. Therefore one of the $(-1)$-components must intersect $C_{\nu_b}$, which in turn implies that the entire string consists only of these two $(-1)$-components.
\end{itemize}

From the above analysis, we see that the string $\st_{ac}^{\leq \nu_a}$ can be contracted to two $(-1)$-components by a sequence of toric blowdowns. The same argument applies to $\st_{bc}^{\geq \nu_b}$. Together with the $N_c$-component, this would produce a toric divisor with self-intersection sequence $(-1,-1,[N_c]^2,-1,-1)$. However, such a toric configuration cannot exist: after further blowing down the two $(-1)$-components adjacent to $N_c$, we would obtain the toric boundary self-intersection sequence $(0,[N_c]^2+2,0)$ on $\CP^2$, which is impossible. Therefore $\nu_b-\nu_a=2$ whenever $[N_c]^2\geq 0$.

\medskip
{\bf Step 5:} If $[N_c]^2=-1$, then $\nu_b-\nu_a=1$.

Suppose $\nu_b-\nu_a=2$. As in Step 4, we may perform toric blowdowns contracting $\st_{ac}^{\leq \nu_a}$ to a string with two adjacent $(-1)$-components, and we denote the sum of their homology classes by $U_a$. Since $U_a\cdot F_b=0$, the light cone lemma implies that the nonzero classes $U_a$ and $F_b$ are proportional. By Step 2, we have $F_b\cdot [C_{\nu_a+1}]=1$, so one of these $(-1)$-components must intersect $C_{\nu_a+1}$. Similarly, $\st_{bc}^{\geq \nu_b}$ can be contracted to a string containing two adjacent $(-1)$-components; denote the sum of their homology classes by $U_b$. Again, one of these components must meet $C_{\nu_a+1}$. Applying the light cone lemma once more, we obtain $U_a=U_b$. Hence both contracted strings must also intersect $N_c$. This produces a toric configuration with self-intersection sequence $(-1,-1,[N_c]^2,-1,-1,[C_{\nu_a+1}]^2)$.
This is impossible, because in $\CP^2\#3\overline{\CP}^2$ with symplectic canonical class $K$,
\[6=K^2=(-1-1+[N_c]^2-1-1+[C_{\nu_a+1}]^2)+12,\]
while $[N_c]^2=-1$ and $[C_{\nu_a+1}]^2\leq -2$. Therefore, $\nu_b-\nu_a=1$ whenever $[N_c]^2=-1$.

\medskip
{\bf Step 6:} Existence of symplectic affine rulings.

It follows from the previous steps together with Lemma \ref{lem:resolutionfiberclass} that, when $[N_c]^2\geq 0$, the common class $F_a=F_b$ is a fiber class. For generic $J\in\mathcal{J}(D)$, this class is represented by an embedded $J$-holomorphic sphere of self-intersection $0$. By its intersection pattern, this sphere meets the unique component $C_{\nu_a+1}=C_{\nu_b-1}$ transversely at a single point and is disjoint from $\st_a\cup\st_b$. This is precisely the desired embedded symplectic sphere. When $[N_c]^2=-1$, after performing successive toric blowups using the common weight sequence of $(p_a,q_a)=(q_b,p_b)$, the resolved class $\tilde{F}_a=\tilde{F}_b$ is represented by an embedded symplectic sphere of self-intersection $0$. By \cite[Proposition 4.3.1]{McduffSiegel}, this sphere corresponds to a symplectic $(p_a,q_a)$-unicuspidal rational curve whose cusp is located at $C_{\nu_a}\cap C_{\nu_b}$ and which is disjoint from the other components of $D$.

\end{proof}

\begin{remark}
  The main point of Proposition \ref{prop:affineruling} is to provide a criterion by $[N_c]^2$ for determining whether the symplectic affine rulings intersecting only the string $\st_c$ are represented by smooth curves or by singular curves. For weighted projective planes, there may be many other affine rulings. Indeed, by reversing the orientations of the strings $\st_{ac}$ and $\st_{bc}$, we can construct rulings that only intersect $\st_a$ or $\st_b$. In the algebraic setting, Daigle--Russell (\cite{DP1,DP2}) gave a combinatorial description that enumerates affine rulings modulo a certain equivalence relation on normal rational surfaces with Picard rank $1$. It would be interesting to develop a symplectic analogue of this result as well.
\end{remark}

\begin{example}
  We conclude with the example of $\CP(11,13,14)$, which exhibits the necessity of considering singular curves for affine rulings intersecting $\st_c$-components rather than only embedded ones. For $(a,b,c)=(11,13,14)$, the resolution strings $\st_a,\st_c,\st_b$ have self-intersection sequences
  \[(-2,-3,-2,-2),\quad (-3,-5),\quad (-2,-2,-2,-2,-2,-3).\]
This is an example with $[N_c]^2=-1$, since in $\CP^2\#12\overline{\CP}^2$ with canonical class $K$,
\[-3=K^2=-2-3-2-2-3-5-2-2-2-2-2-3+[N_a]^2+[N_b]^2+[N_c]^2+30.\]
We have $\nu_a=1$ and $\nu_b=2$. The strings $\st_{ac}^{\leq 1}$ and $\st_{bc}^{\geq 2}$ have self-intersection sequences
\[(-2,-3,-2,-2,-1,-3),\qquad (-3,-2,-2,-2,-2,-2,-1,-5),\]
respectively, and their corresponding determinant $\Delta_i$-sequences from Section \ref{section:resfiberclass} are
\[(1,2,5,8,11,3,-2),\qquad (1,3,5,7,9,11,13,2,-3).\]
It follows that the common class $F_a=F_b$ is represented by a symplectic $(2,3)$-unicuspidal rational curve with cusp located at the intersection of the $(-3)$- and $(-5)$-components of $\st_c$. See Figure \ref{fig:111314}.
\end{example}

\begin{figure}[h]
\begin{center}
\begin{tikzpicture}[x=1cm,y=1cm,line cap=round,line join=round,>=Latex]
  \tikzset{
    string/.style={draw=black!70, line width=0.8pt},
    hidden/.style={draw=black!55, dashed, line width=0.8pt},
    cusp/.style={draw=red!85!brown, line width=1pt}
  }

  \draw[string] (0.0,0) -- (1,-0.5);
  \draw[string] (0.8,-0.5) -- (1.8,0.0);
  \draw[string] (1.6,0.0) -- (2.6,-0.5);
  \draw[string] (2.4,-0.5) -- (3.4,0.0);
  \node at (0.5,-0.2) {$-2$};
  \node at (1.2,-0.2) {$-3$};
  \node at (2.2,-0.2) {$-2$};
  \node at (2.9,-0.2) {$-2$};
  \node at (1.75,-0.92) {$\st_a$};

  \draw[hidden] (3.1,-0.1) -- (4.8,-0.1);

  \node[above] at (3.95,-0.08) {$N_b$};

  \draw[string] (4.5,0.0) -- (5.5,-0.5);
  \draw[string] (5.3,-0.5) -- (6.3,0.0);
  \node at (4.8,-0.3) {$-3$};
  \node at (6,-0.3) {$-5$};
  \node at (5.55,-0.92) {$\st_c$};

  \draw[hidden] (6.0,-0.1) -- (7.7,-0.1);
  \node[above] at (6.85,-0.08) {$N_a$};

  \draw[string] (7.4,0.0) -- (8.4,-0.5);
  \draw[string] (8.2,-0.5) -- (9.2,0.0);
  \draw[string] (9,0.0) -- (10.0,-0.5);
  \draw[string] (9.8,-0.5) -- (10.8,0.0);
  \draw[string] (10.6,0.0) -- (11.6,-0.5);
  \draw[string] (11.4,-0.5) -- (12.4,0.0);
  \node at (7.9,-0.2) {$-2$};
  \node at (8.6,-0.2) {$-2$};
  \node at (9.6,-0.2) {$-2$};
  \node at (10.3,-0.2) {$-2$};
  \node at (11.3,-0.2) {$-2$};
  \node at (12,-0.2) {$-3$};
  \node at (9.75,-0.92) {$\st_b$};

  \draw[cusp] (4.4,0.4) .. controls (4.8,0.2) and (5.1,0) .. (5.4,-0.45);
  \draw[cusp] (6.4,0.4) .. controls (6,0.2) and (5.7,0.) .. (5.4,-0.45);
  \node[above] at (5.5,0.2) {$(2,3)$-cusp};
\end{tikzpicture}
\caption{The red curve represents a symplectic $(2,3)$-cuspidal affine ruling in class $F_a=F_b$ on the minimal resolution of $\CP(11,13,14)$.}\label{fig:111314}
\end{center}
\end{figure}
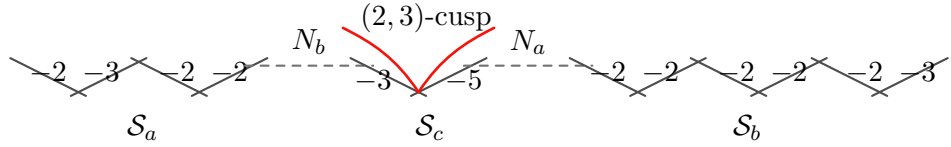

\subsection{Recovering sub-toric configurations from affine rulings}
We now prove the converse direction of Theorem \ref{thm:ruling}. We begin by adapting the argument of Section \ref{section:connecting} to show that three strings can still be joined into a connected configuration by exceptional spheres when the symplectic area assumption is replaced by the existence of affine ruling. 

\begin{lemma}\label{lemma:connecting}
   Let $D=\st_1\cup\st_2\cup\st_3$ be disjoint symplectic strings in $(X,\omega)$ with $l(\st_1)+l(\st_2)+l(\st_3)=b_2^-(X)$. Assume that $A$ is a component of $\st_3$ such that every component of $D\setminus A$ has self-intersection at most $-2$. If there exists an embedded symplectic sphere $C$ of self-intersection $0$ that intersects $A$ transversely at a single point and is disjoint from all other components of $D$, then there exist symplectic exceptional spheres $C_1,C_2$ such that
   \begin{itemize}[nosep]
    \item $D\cup C_1\cup C_2$ is a symplectic divisor without non-toric exceptional spheres;
    \item $[C_1]\cdot [\st_1]=[C_1]\cdot [\st_3]=[C_2]\cdot [\st_2]=[C_2]\cdot [\st_3]=1$, $[C_1]\cdot[\st_2]=[C_2]\cdot[\st_1]=[C_1]\cdot [C_2]=0$.
   \end{itemize}
\end{lemma}

\begin{proof}
  After a small perturbation, we may assume that $D\cup C$ is a symplectic divisor. Choose any $J\in \mathcal{J}(D\cup C)$ such that the fiber class $[C]$ is $J$-nef. Let $\Theta_1,\Theta_2$ be the $J$-holomorphic subvarieties in class $[C]$ passing through points on $\st_1$ and $\st_2$, respectively. We claim that $\Theta_1\neq \Theta_2$. Indeed, otherwise there would exist a component $B\in \Theta_1=\Theta_2$ connecting $\st_1$ to $\st_2$. By Proposition \ref{prop:fiberclass}(2) and our assumption on the self-intersection numbers of the components of $D\setminus A$, the configuration $(D\setminus A)\cup B$ would then be negative definite. However, it has $b_2^-(X)$ components, contradicting Lemma \ref{lem:fibercomplement}.

  Therefore, by Proposition \ref{prop:fiberclass}(3), we can choose two disjoint symplectic exceptional spheres $C_1,C_2$ contained in $\Theta_1,\Theta_2$, respectively. Since the configurations of $\Theta_1$ and $\Theta_2$ are trees of embedded spheres by Proposition \ref{prop:fiberclass}(1), after another small perturbation we may assume that $D\cup C_1\cup C_2$ is a symplectic divisor. We now verify the stated intersection relations below.
  
  First, since $\Theta_1\neq \Theta_2$, we naturally have
  \[
    [C_1]\cdot[\st_2]=[C_2]\cdot[\st_1]=[C_1]\cdot [C_2]=0.
  \]
  Next, we claim that $C_1$ (respectively, $C_2$) must intersect both $\st_1$ and $\st_3$ (respectively, $\st_2$ and $\st_3$). Indeed, if $C_1$ did not intersect both $\st_1$ and $\st_3$, then $\Theta_1$ would contain components other than $C_1\cup \st_1$, and hence $(D\setminus A)\cup C_1$ would be negative definite by Proposition \ref{prop:fiberclass}(2), again contradicting Lemma \ref{lem:fibercomplement}. The same argument applies to $C_2$. Since the configurations of $\Theta_1$ and $\Theta_2$ are trees and hence contain no loops, it follows that
  \[
    [C_1]\cdot[\st_1]=[C_2]\cdot[\st_2]=1.
  \]
  Finally, we show that $[C_1]\cdot [\st_3]\leq 1$; the same proof gives $[C_2]\cdot [\st_3]\leq 1$. Let $\st_3',\st_3''$ denote the two (possibly empty) strings of connected components of $\st_3\setminus A$, and let $\Theta'_3,\Theta''_3$ be the $J$-holomorphic subvarieties passing through $\st_3',\st_3''$, respectively. There are two cases.
  \begin{itemize}[nosep]
   \item If $\Theta_1\neq \Theta_3'$ and $\Theta_1\neq \Theta_3''$, then the configuration of $\Theta_1$ cannot contain more components than $\st_1\cup C_1$ by Lemma \ref{lem:fibercomplement}. This will imply that $C_1$ must intersect the $A$-component. However, since $C_1$ is the only $(-1)$-component in $\Theta_1$, Proposition \ref{prop:fiberclass}(3) indicates that the multiplicity of $C_1$ is at least $2$. This contradicts $[C]\cdot A=1$.
    \item If $\Theta_1=\Theta_3'$ (similarly, if $\Theta_1=\Theta_3''$), then Proposition \ref{prop:fiberclass}(1) gives $[C_1]\cdot [\st_3']=1$ since there is no loop, while $[C_1]\cdot [A]=0$ because $[C]\cdot [A]=1$, and $[C_1]\cdot [\st_3'']=0$ because $\Theta_3'\neq \Theta_3''$.
  \end{itemize}
  Therefore, we must have $[C_1]\cdot [\st_3]=[C_1]\cdot([\st_3']+[\st_3'']+[A])=1$. By symmetry, also $[C_2]\cdot [\st_3]=1$, and hence
  \[
    [C_1]\cdot [\st_1]=[C_1]\cdot [\st_3]=[C_2]\cdot [\st_2]=[C_2]\cdot [\st_3]=1.
  \]
  It remains to show that $D\cup C_1\cup C_2$ admits no non-toric exceptional spheres. By the preceding analysis, after possibly relabeling we may assume that $\Theta_1=\Theta_3'$ and $\Theta_2=\Theta_3''$. Suppose, for contradiction, that there exists a non-toric exceptional sphere $C_e$. Then the configuration $C_e\cup (D\setminus A)$ must be negative definite. This is immediate if $C_e$ intersects the $A$-component. If instead $C_e$ intersects a component of $\st_3'$ (and similarly for $\st_3''$), then $\st_1\cup C_e\cup\st_3'$ is a proper sub-configuration of $\Theta_1$, so negative definiteness again follows from Proposition \ref{prop:fiberclass}(2). Since $C_e\cup (D\setminus A)$ is disjoint from $C$ and has $b_2^-(X)$ components, this contradicts Lemma \ref{lem:fibercomplement}.

\end{proof}

\begin{proposition}\label{prop:recoverfromruling}
  Let $D=\st_1\cup\st_2\cup\st_3$ be three disjoint symplectic HJ strings in $(X,\omega)$ where $[\omega]\cdot K_\omega<0$. Suppose $D$ is full and of $(a,b,c)$-type. If there exists $C\subseteq (X,\omega)$ which is
\begin{itemize}[nosep]
  \item either a symplectic $(p,q)$-unicuspidal rational curve with $[C]^2=pq$, whose $(p,q)$-cusp is at the intersection of two components of one string $\st_i$ and disjoint from other components in $D$;
  \item or an embedded symplectic sphere of self-intersection $0$ intersecting exactly one component in $D$ transversely,
\end{itemize} 
then $D$ is sub-toric.
\end{proposition}

\begin{proof}
  We may assume that $C$ intersects $\st_3$ after relabeling. First, observe that the hypotheses of Lemma \ref{lemma:connecting} can always be recovered. If $C$ is embedded, this follows directly from the HJ assumption on $\st_1,\st_2,\st_3$. If $C$ is $(p,q)$-unicuspidal, then we proceed as in Section \ref{section:resfiberclass} by performing successive toric blowups determined by the weight sequence $\mathcal{W}(p,q)=(m_1,\cdots,m_L)$. Let $e_1,\cdots,e_L$ denote the corresponding exceptional classes. Then the total transform of $D$ is again a disjoint union of three strings, now with total length $b_2^-(X)+L$, and all components other than the $e_L$-component have self-intersection at most $-2$. Moreover, the resolution class $[C]-\sum_{i=1}^L m_i e_i$ is a fiber class in $(X\#L\overline{\CP}^2,\tilde{\omega})$, represented by an embedded symplectic sphere $\tilde{C}$ that intersects only the $e_L$-component of the total transform of $D$. Therefore, by Lemma \ref{lemma:connecting}, there exist exceptional spheres $C_1$ and $C_2$ in $(X,\omega)$ or in its blowup $(X\#L\overline{\CP}^2,\tilde{\omega})$ such that, by denoting the total transform of $D$ still by $\tilde{D}$ ($\tilde{D}=D$ if $L=0$), $\tilde{D}\cup C_1\cup C_2$ forms a connected symplectic divisor.
  
Next, consider the class
\[
  U:=-K_{\tilde{\omega}}-[\tilde{D}]-[C_1]-[C_2].
\]
We claim that $[\tilde{\omega}]\cdot U>0$. By the intersection pattern in Lemma \ref{lemma:connecting},
\[
  U^2=(-K_{\tilde{\omega}}-[\tilde{D}])^2+(-[C_1]-[C_2])^2+2(-K_{\tilde{\omega}}-[\tilde{D}])\cdot(-[C_1]-[C_2])
  =(K_{\tilde{\omega}}+[\tilde{D}])^2+2.
\]
Since $D$ is full and of $(a,b,c)$-type, Corollary \ref{cor:3nminus6} implies that $(K_{\omega}+[D])^2\geq -2$. Because toric blowup preserves the adjoint class, it follows that $(K_{\tilde{\omega}}+[\tilde{D}])^2\geq -2$ as well. Hence $U^2\geq -1$.

If $U^2\geq 0$, then, because $C_1$ and $C_2$ lie in the complement of $\tilde{C}$,
\[
  U\cdot [\tilde{C}]=(-K_{\tilde{\omega}}-[\tilde{D}])\cdot [\tilde{C}]=2-1=1.
\]
Therefore $U$ and $[\tilde{C}]$ lie in the same connected component of the positive cone in $H_2(X\#L\overline{\CP}^2;\Z)$. Since $[\tilde{\omega}]\cdot[\tilde{C}]>0$, we conclude that $[\tilde{\omega}]\cdot U>0$.

If $U^2=-1$, then, because the configuration $\tilde{D}\cup C_1\cup C_2$ is connected and loop-free, adjunction gives
\begin{align*}
  K_{\tilde{\omega}}\cdot U&=(-U-[\tilde{D}]-[C_1]-[C_2])\cdot U=-U^2-([\tilde{D}]+[C_1]+[C_2])\cdot U\\
  &=1-([\tilde{D}]+[C_1]+[C_2])\cdot (-K_{\tilde{\omega}}-[\tilde{D}]-[C_1]-[C_2])\\
  &=1-2=-1.
\end{align*}
It then follows from \cite[Theorem A]{LiLiu96} that $U\in \mathcal{E}_{\tilde{\omega}}$, and hence again $[\tilde{\omega}]\cdot U>0$.
 
Therefore, $\overline{\overline{D}}:=\tilde{D}\cup C_1\cup C_2$ is a connected symplectic divisor satisfying $[\tilde{\omega}]\cdot(K_{\tilde{\omega}}+[\overline{\overline{D}}])<0$. Hence we may carry out the quasi-minimal reduction for $\overline{\overline{D}}$ as in Section \ref{sec:LN25}. By Lemmas \ref{lemma:connecting} and \ref{lem:blowupexceptional}, this reduction uses only toric and half-toric blowdowns. We claim that $\overline{\overline{D}}$ has no branching component.

Suppose on the contrary that $\overline{\overline{D}}$ has a branching component $B$. By Remark \ref{rmk:-1reduction}, the argument of Lemma \ref{lem:-1reduction} still shows that one connected component of $\overline{\overline{D}}\setminus B$ can be contracted to a $(-1)$-sphere through a sequence of toric and half-toric blowdowns. Then, exactly as in the proof of Proposition \ref{prop:isstring}, every branching component of $\overline{\overline{D}}$ must lie on $\st_3$, and $\st_1$ and $\st_2$ must be $(-2)$-strings. After permuting $(a,b,c)$, we may therefore assume that $l(\st_1)=a-1$ and $l(\st_2)=b-1$. By Lemma \ref{lem:-2chains}, the string $\st_3$ has self-intersection sequence $(-a,-k,-b)$ for some $k\ge 2$, $(1-a,1-b)$ when $k=1$ and $b\neq 2$, or $(2-a)$ when $k=1$ and $b=2$. Performing $a-1$ toric blowdowns followed by one half-toric blowdown contracts $\st_1\cup C_1$, and similarly $b-1$ toric blowdowns followed by one half-toric blowdown contract $\st_2\cup C_2$. Thus $\overline{\overline{D}}$ can be reduced to a single string $\check{\st}_3$.

If $\tilde{D}=D$, so that $C$ is an embedded symplectic sphere, the rest of the argument is identical to that of Proposition \ref{prop:isstring}. In this case the branching component can only be the $(-k)$-component of $\st_3$, and $\check{\st}_3$ is a $3$-component string in a Hirzebruch surface. We can use Claim \ref{claim:sphereinHirzebruch} to obtain contradiction.

Assume now that $\tilde{D}\neq D$, so that $C$ is cuspidal. Let $\tilde{\st}_3\subseteq \tilde{D}$ be the total transform of $\st_3$ under the successive toric blowups. Then $\check{\st}_3$ can also be viewed as being obtained from $\tilde{\st}_3$ by performing $a+b$ non-toric blowdowns at the two distinct components to which $C_1$ and $C_2$ are attached. After deleting the $e_L$-component, the string $\check{\st}_3$ contains two non-negative definite strings $\check{\st}_3'$ and $\check{\st}_3''$. Aside from the two components where non-toric blowdowns occur, every component of $\check{\st}_3'$ and $\check{\st}_3''$ has self-intersection at most $-2$. Lemma \ref{lem:adjacent-1} therefore implies that each of $\check{\st}_3'$ and $\check{\st}_3''$ can be reduced, using only toric blowdowns, to a length-$2$ string of two $(-1)$-components. Consequently, $\check{\st}_3$ itself can be reduced by toric blowdowns to a string $\st_3'''$ with self-intersection sequence $(-1,-1,-1,-1,-1)$, where the middle $(-1)$ corresponds to the $e_L$-component.

We now consider the $\Xi$-invariant introduced in the proof of Proposition \ref{prop:secondexceptional}. In the case where $\st_3$ has self-intersection sequence $(-a,-k,-b)$, we have $\Xi(\st_3)=a+b+k-9$, and hence also $\Xi(\tilde{\st}_3)=a+b+k-9$ because toric blowups preserve $\Xi$. Contracting $\st_1\cup C_1$ and $\st_2\cup C_2$ increases the self-intersection numbers of two components of $\tilde{\st}_3$ by $a$ and $b$, respectively, so $\Xi(\check{\st}_3)=k-9$. On the other hand, the previous paragraph shows that $\check{\st}_3$ reduces to $\st_3'''$, so $\Xi(\check{\st}_3)=\Xi(\st_3''')=-10$. This is impossible because $k\ge 2$. The remaining cases lead to contradictions in the same way, by considering $k=1$ and $k=1$ with $b=2$. Therefore $\overline{\overline{D}}$ has no branching components.
  
Finally, once we know $\overline{\overline{D}}$ is a string, we can employ the same argument as in the proof of Proposition \ref{prop:cycle} to represent the class $U$ by an embedded symplectic sphere $C_U$, so that $\overline{\overline{D}}$ becomes a toric symplectic log Calabi--Yau divisor. This proves that $D$ is sub-toric.
\end{proof}

\begin{remark}
The proof of the above proposition contains two new features in addition to the overlaps with the argument in Section \ref{Section:HJstrings}. First, without the gap-admissible condition, one must use the existence of symplectic affine rulings to show that the class $U$ has positive symplectic area. Second, because we allow affine rulings given by unicuspidal curves, one must pass to a resolution in order to obtain a fiber class and work with the total transform $\tilde{D}$. The connecting exceptional spheres $C_1$ and $C_2$ provided by Lemma \ref{lemma:connecting} may intersect components of $\tilde{D}$ that do not belong to the proper transform of $D$.
\end{remark}

\begin{proof}[Proof of Theorem \ref{thm:ruling}]
  The theorem follows immediately from Propositions \ref{prop:affineruling} and \ref{prop:recoverfromruling}.
\end{proof}

\bibliographystyle{plain}
\bibliography{mybib}

\end{document}